\documentclass[12pt]{extarticle}
\usepackage{amsmath, amsthm, amssymb, mathtools, hyperref, color}
\usepackage[shortlabels]{enumitem}
\usepackage{graphicx}
\usepackage[all]{xypic}
\usepackage{makecell}

\oddsidemargin \evensidemargin
\newtheorem{theorem}{Theorem}
\numberwithin{theorem}{section}
\newtheorem{proposition}[theorem]{Proposition}
\newtheorem{lemma}[theorem]{Lemma}
\newtheorem{corollary}[theorem]{Corollary}

\newtheorem{conjecture}[theorem]{Conjecture}

\newtheorem{remark}[theorem]{Remark}
\newtheorem{example}[theorem]{Example}

\newcommand{\RR}{\mathbb{R}}
\newcommand{\QQ}{\mathbb{Q}}
\newcommand{\PP}{\mathbb{P}}
\newcommand{\CC}{\mathbb{C}}

 \date{}
 
\title{\textbf{Decomposing Tensors into Frames}
\author{Luke Oeding, Elina Robeva and Bernd Sturmfels}}

\begin{document}
\maketitle

\begin{abstract} \noindent
A symmetric tensor of small rank decomposes into a configuration of only few vectors.
We study the variety of tensors for which  this configuration is a unit norm tight frame.
\end{abstract}

\section{Introduction}

A fundamental problem in computational algebraic geometry,
with a wide range of applications, is the  
low rank decomposition of symmetric tensors; see e.g.~\cite{AGHKT, BCMT, CGLM, OO, Rob}.
If $T = (t_{i_1 i_2 \cdots i_d})$ is a symmetric tensor 
in ${\rm Sym}_d(\CC^n)$, then such a decomposition takes the form
\begin{equation}
\label{eq:Trank1}
 T \,=\, \sum_{j=1}^r \lambda_j {\bf v}_j^{\otimes d} . 
 \end{equation}
Here $\lambda_j \in \CC$ and ${\bf v}_j = (v_{1j}, v_{2j}, \ldots, v_{nj}) \in \CC^n$ for
$j=1,2,\ldots,r$.
The smallest $r$ for which a representation (\ref{eq:Trank1}) exists
is the {\em rank} of $T$.  
In particular,  each ${\bf v}_j^{\otimes d} $ is a tensor of rank $1$.

An equivalent way to represent a symmetric tensor $T$ is 
as the homogeneous polynomial
\begin{equation}
\label{eq:generalp}
 T \,\,\, = \sum_{i_1,\ldots,i_d=1}^n 
t_{i_1 i_2 \cdots i_d} \cdot x_{i_1} x_{i_2} \cdots x_{i_d}.
\end{equation}
If $d=2$, then \eqref{eq:generalp} is the identification of symmetric matrices with quadratic forms.
Written as a polynomial, the right hand side of (\ref{eq:Trank1}) is a linear combination of powers of linear forms:
\begin{equation}
\label{eq:Trank2}
T \,\,\, = \,\,\,\sum_{j=1}^r \lambda_j (v_{1j} x_1 + v_{2j} x_2 + \cdots + v_{nj} x_n)^d.
\end{equation}
The decomposition in \eqref{eq:Trank1} and \eqref{eq:Trank2} is called {\em Waring decomposition}. 
When $d=2$, it corresponds to orthogonal diagonalization of symmetric matrices.
We could subsume the constants $\lambda_i$ into the vectors
${\bf v}_i$ but we prefer to leave (\ref{eq:Trank1}) and (\ref{eq:Trank2}) as is,
   for reasons    to be seen shortly. The (projective) variety of all such symmetric tensors  
is the {\em $r$-th secant variety of the Veronese variety}.
The vast literature on the geometry and equations 
of this variety (cf.~\cite{Lan}) forms the mathematical
foundation for low rank decomposition algorithms for symmetric tensors.

In many situations one places further restrictions on the
summands in (\ref{eq:Trank1}) and (\ref{eq:Trank2}), such as being real and nonnegative.
Applications to machine learning in \cite{AGHKT} concern the case when $r=n$ and the 
 vectors ${\bf v}_1,\ldots,{\bf v}_n$ form an orthonormal basis of $\RR^n$.
 The article \cite{Rob} characterizes the {\em odeco variety}  of all tensors  that admit such
  an orthogonal decomposition. 
  
  The present paper takes this one step further by connecting 
    tensors  to {\em frame theory} \cite{CStr, CMS, CKP, DS, Str}.
  We examine the scenario when the ${\bf v}_j$ form a {\em  \bf f}inite {\bf u}nit {\bf n}orm {\bf t}ight {\bf f}rame (or~funtf) of
$\mathbb{R}^n$, an object of  recent interest at the
interface of applied functional analysis and algebraic geometry.
Consider a configuration $V = ({\bf v}_1,\ldots,{\bf v}_r) \in (\RR^n)^r$ 
of $r$ labeled vectors in $\RR^n$. We also regard this as an $n \times r$-matrix $V = (v_{ij})$.
We call $V$ a {\em funtf }  if
\begin{equation}
\label{eq:funtf}
 V \cdot V^T = \,\,\frac{r}{n} \cdot {\rm Id}_n 
\qquad \hbox{and} \qquad
\sum_{j=1}^n v_{ij}^2 = 1 \quad \hbox{for}  \,\,\,i = 1,2,\ldots,r. 
\end{equation}
This is an inhomogeneous system  of  $n^2+r$ 
quadratic equations in $nr$ unknowns.
The {\em funtf variety}, denoted ${\mathcal F}_{r,n}$ as in \cite{CMS},
is the subvariety of complex affine space
$\CC^{n \times r}$ defined by (\ref{eq:funtf}).
For the state of the art we refer to the article \cite{CMS} by
 Cahill, Mixon and Strawn, and the references therein.
 A detailed review, with some new perspectives,
 will be given in   Section~\ref{sec2}.

We homogenize the funtf variety by attaching a scalar 
$\lambda_i$ to each vector ${\bf v}_i$.
The result maps into  the  projective space 
$\mathbb{P}({\rm Sym}_d(\CC^n)) = \mathbb{P}^{\binom{n-1+d}{d}-1}$ of symmetric tensors,
via the formulas (\ref{eq:Trank1}) and (\ref{eq:Trank2}).
Our aim is to study the closure of the image
of that map. This is denoted $\mathcal{T}_{r,n,d}$.
We call it the {\em variety of {\bf fra}me {\bf deco}mposable tensors},
or the {\em fradeco variety}.
Here $r,n,d$ are positive integers with $r \geq n$.
For $r=n$, $\mathcal T_{n, n, d}$ is the odeco variety of~\cite{Rob}.

\begin{example} 
\label{ex:opening} \rm 
Let $n= 3, d=4$, and consider  the symmetric $3 {\times} 3 {\times} 3 {\times} 3$-tensor 
\begin{equation}
\label{eq:ex434a}
 \begin{matrix} T & = &  59 (x_1^4+x_2^4+x_3^4)
\, -\, 16 (x_1^3 x_2 +x_1 x_2^3 + x_1^3 x_3 + x_2^3 x_3 + x_1 x_3^3 + x_2 x_3^3) \\ & &
 +\,\, 66 (x_1^2 x_2^2 + x_1^2 x_3^2 + x_2^2 x_3^2)
\, +\, 96 (x_1^2 x_2 x_3 + x_1 x_2^2 x_3 +  x_1 x_2 x_3^2). \end{matrix}
\end{equation}
This ternary quartic lies in $\mathcal{T}_{4,3,4}$, i.e.~this tensor
has {\em fradeco rank} $\,r=4$.
To see this, note~that
\begin{equation}
 \label{eq:ex434b}
T \,= \,
  \frac{1}{12} (-5x_1+x_2+x_3)^4
+ \frac{1}{12} (x_1-5x_2+x_3)^4
+ \frac{1}{12} (x_1+x_2-5x_3)^4 
+ \frac{1}{12} (3x_1 + 3x_2 + 3x_3)^4.
\end{equation}
The corresponding four vectors, appropriately scaled, form a
finite unit norm tight frame:
\begin{equation}
\label{eq:appscaled}
 V \quad = \quad
\frac{1}{3 \sqrt{3}} \begin{pmatrix}
-5 & \phantom{-}1 & \phantom{-}1 \,& \,3\, \,\\
\phantom{-}1 & -5 & \phantom{-}1 \,& \,3\,\, \\
\phantom{-}1 & \phantom{-}1 & -5  \,& \,3\,\, \end{pmatrix} \,\, \in \,\, \mathcal{F}_{4,3}.
\end{equation}
The fradeco variety $\mathcal{T}_{4,3,4}$ is a projective variety of dimension 
$6$ and degree $74$ in $\PP^{14} $. It 
is parametrized by applying rotation matrices $\rho \in {\rm SO}_3$ to
all ternary quartics of the form
\begin{equation}
 \label{eq:ex434c}
T \,= \,
  \lambda_1 (-5x_1+x_2+x_3)^4
+ \lambda_2 (x_1-5x_2+x_3)^4
+ \lambda_3 (x_1+x_2-5x_3)^4 
+ \lambda_4 (3x_1 + 3x_2 + 3x_3)^4.
\end{equation}
The title of our paper refers to the task
of finding the output (\ref{eq:ex434b}) from the input
   (\ref{eq:ex434a}).   In this particular case, the decomposition  can be found easily 
using Sylvester's classical {\em Catalecticant Algorithm}, as explained in
\cite[Section 2.2]{OO}. In general, this will be more difficult to do.
\hfill $\diamondsuit$
\end{example}

The {\em fradeco rank} of a symmetric tensor $T \in {\rm Sym}_d(\RR^n)$
is defined as the smallest $r$ such that $T \in \mathcal{T}_{r,n,d}$. This
property does not imply
that $T$ also has a frame decomposition (\ref{eq:Trank1})
of length $r+1$. Indeed,  we often have
$\mathcal{T}_{r,n,d} \not\subset \mathcal{T}_{r+1,n,d}$.
For instance, the odeco quartic
$ x_1^4 + x_2^4 + x_3^4$ lies in
$\,\mathcal{T}_{3,3,4} \backslash \mathcal{T}_{4,3,4}$,
by the constraint in Example \ref{ex:T434}.
See also Example \ref{ex:notcontain}.

\smallskip

This paper is organized as follows. 
In Section \ref{sec2} we give an introduction to the
algebraic geometry of the funtf variety $\mathcal{F}_{r,n}$.
This lays the foundation for the subsequent study of fradeco tensors.
Section \ref{sec3} is concerned with 
the case of symmetric $2 {\times} 2 {\times}  \cdots {\times} 2$-tensors $T$.
These correspond to binary forms $(n=2)$. We characterize frame
decomposable tensors in terms of rank conditions on matrices.
In Section \ref{sec4} we investigate the general case $n \geq 3$,
and we present what we know about the fradeco varieties $\mathcal{T}_{r,n,d}$.
Section \ref{sec5} is devoted to numerical algorithms for studying
$\mathcal{T}_{r,n,d}$ and for decomposing its elements into frames.

\section{Finite unit norm tight frames}
\label{sec2}

In this section we discuss various representations of the funtf variety
$\mathcal{F}_{r,n}$. This may serve as an invitation
 to the emerging interaction between 
algebraic geometry and frame theory.

Each variety studied in this paper is defined over the real field $\RR$
and is the Zariski closure of its set of real points. This Zariski
closure lives in affine or projective space over $\CC$.
For instance, ${\rm SO}_n$ is the group of $n {\times} n$ rotation matrices
$\rho$, and such matrices have entries in~$\RR$. However, when 
referring to ${\rm SO}_n$ as an algebraic variety  we mean the
irreducible subvariety of $\CC^{n \times n}$ defined  by the 
polynomial equations $\,\rho \cdot \rho^T = {\rm Id}_n$ and ${\rm det}(\rho) = 1$.
Likewise, a funtf $V$ is a real $n \times r$
matrix, but the funtf variety $\mathcal{F}_{r,n}$ lives in $\CC^{n \times r}$.
It consists of all complex solutions to the quadratic
equations (\ref{eq:funtf}). In the frame theory literature \cite{CStr, CMS,DS,  Str}
there is also a complex Hermitian version of $\mathcal{F}_{r,n}$,
but it will not be considered in this paper.

 It is important  to distinguish $\mathcal{F}_{r,n}$ from the variety of
 {\em Parseval frames}, here denoted $\mathcal{P}_{r,n}$.
 The latter is much easier than the former. The variety $\mathcal{P}_{r,n}$ is
    defined by the matrix equation 
$$ V \cdot V^T = \,\, {\rm Id}_n . $$
The real points on $\mathcal{P}_{r,n}$ are smooth and Zariski dense,
and they form the Stiefel manifold of all orthogonal projections 
$\RR^r \rightarrow \RR^n$.  
Hence $\mathcal{P}_{r,n}$ is irreducible of dimension 
$nr - \binom{n+1}{2}$.

One feature that distinguishes $\mathcal{P}_{r,n}$ from $\mathcal{F}_{r,n}$
is the existence of a canonical map  $\mathcal{P}_{r,n+1} \rightarrow \mathcal{P}_{r,n}$.
Indeed, by Naimark's Theorem \cite{CFMPS}, every Parseval frame is the orthogonal projection of
an orthonormal basis of $\RR^r$, so we can add a row to $V \in \mathcal{P}_{r,n}$
and get a matrix in $\mathcal{P}_{r,n+1}$.
There is no analogous statement for the variety $\mathcal{F}_{r,n}$.
We begin with the following result.

\begin{theorem} \label{thm:strawn}
The dimension of the funtf variety $\mathcal{F}_{r,n}$ is
\begin{equation}
\label{eq:Fdim}
 \qquad {\rm dim}( \mathcal{F}_{r,n}) \,\, = \,\, (n-1)\cdot ( r - \frac{n}{2} - 1)
\qquad \hbox{provided}\,\, \,r  > n \geq 2.
\end{equation}
It is irreducible when $r \geq n+2 > 4$.
\end{theorem}

\begin{proof}
Cahill, Mixon and Strawn  \cite[Theorem 1.4]{CMS} proved that
$\mathcal{F}_{r,n}$ is irreducible when $r \geq n+2 > 4$.
The dimension formula comes from two articles: one by Dykema and Strawn \cite[Theorem 4.3(ii)]{DS}
regarding the case when $r$ and $n$ are relatively prime, and one by Strawn \cite[Corollary 3.5]{Str} which studies the local geometry for all $r,n$. In these articles it is shown that the real points in $\mathcal{F}_{r,n}$ have a dense open subset that forms a manifold of dimension $(n-1)\cdot ( r - \frac{n}{2} - 1)$.
The arguments in \cite{CMS} show that the real points are Zariski dense in
the complex  variety $\mathcal{F}_{r,n}$. Hence (\ref{eq:Fdim}) is 
the correct formula for the dimension of $\mathcal{F}_{r,n}$.
\end{proof}

Next to the dimension, the most important invariant of an algebraic
variety is its {\em degree}. By this we mean the degree of its 
projective closure \cite[\S 8.4]{CLO}.
This can be computed using symbolic software for Gr\"obner bases,
or using numerical algebraic geometry software.
The dimension and degree of $\mathcal{F}_{r,n}$ for
small $r,n$ in Table \ref{tab:one}
 were computed using {\tt Bertini}~\cite{BertiniSoftware}.

\begin{table}
\[
\begin{array}{|c|c|c|c|c|}
 \hline
r & n &  \dim \mathcal{F}_{r,n} & \deg \mathcal{F}_{r,n} & \text{\# components \& degrees  }
 \\ \hline
3 & 2 & 1 & 8\cdot 2 & 8\text{ components, each degree }2  \\ \hline
4 & 2 & 2 & 12\cdot 4 & 12\text{ components, each degree }4  \\ \hline
5 & 2 & 3 & 112 &\text{ irreducible } \\ \hline
6 & 2 & 4 & 240 &\text{ irreducible }  \\ \hline
7 & 2 & 5 & 496 & \text{ irreducible }  \\ \hline
4 & 3 & 3 & 16\cdot 8 & 16 \text{ components, each degree } 8  \\ \hline
5 & 3 & 5 & 1024 & \text{ irreducible }  \\ \hline
6 & 3 & 7 & 2048 & \text{ irreducible }  \\ \hline
7 & 3 & 9 & 4096 & \text{ irreducible }  \\ \hline
5 & 4 & 6 & 32\cdot 40 & 32\text{ components, each degree }40  \\ \hline
6 & 4 & 9 & 20800 & \text{ irreducible }  \\ \hline
7 & 4 & 12 & 65536 & \text{ irreducible }  \\ \hline
\end{array}
\]
\caption{\label{tab:one}
Dimension and degree of the funtf variety in some small cases}
\end{table}

The case $r=n+1$ is special. Here, the funtf variety decomposes
into $2^{n+1}$ irreducible components, each of which is
affinely isomorphic to the $\binom{n}{2}$-dimensional variety ${\rm SO}_n$.
This will be explained in Corollary \ref{cor:quotientspace}.
The next example discusses  one other exceptional case.

\begin{example}[$r=4,n=2$] \label{eq:vierzwei}
\rm
Following (\ref{eq:funtf}), the defining ideal of the funtf variety $\mathcal{F}_{4,2}$ equals
$$  \langle \, 
v_{11}^2 + v_{12}^2 + v_{13}^2 + v_{14}^2-2, \,
v_{11} v_{21} + v_{12} v_{22} + v_{13} v_{23} + v_{14} v_{24} 
 \rangle\, + \,\langle \,v_{i1}^2 + v_{i2}^2-1 \,: \,i=1,2,3,4\, \rangle. $$
 Note that this contains $ v_{21}^2 + v_{22}^2 + v_{23}^2 + v_{24}^2-2$.
 Using Gr\"obner basis software, such as {\tt Macaulay2} \cite{M2},
one checks that this ideal is the complete intersection of the six given quadrics,
it is radical, and its degree is $48$.  Primary decomposition reveals
that this ideal  is the intersection of $12$ prime ideals, each of degree $4$.
One of these associated primes~is 
$$  \langle \,
 v_{11} - v_{22}, v_{12}+v_{21},\,
 v_{31} - v_{42}, v_{32}+v_{41}\,
 \rangle\, + \,\langle \,v_{i1}^2 + v_{i2}^2-1 \,: \,i=1,2,3,4\, \rangle. $$
 The irreducible variety of this particular prime ideal consists of the  $2\times 4$-matrices 
$$ V \quad = \quad \begin{pmatrix} R_1 \mid R_2 \end{pmatrix}, $$
where $R_1$ and $R_2$ are rotation matrices of format $2 \times 2$.
The other $11$ components are obtained by replacing 
$R_i$ by $-R_i$ and permuting columns. The image of $V$ under the
map to binary forms is a linear combination of two odeco forms,
one given by $R_1$ and the other  by $R_2$. 
\hfill $\diamondsuit$
\end{example}

The real points of   $\mathcal{F}_{r,n}$ live
in $(\mathbb{S}^{n-1})^r$ where
$\mathbb{S}^{n-1} = \{u \in \mathbb{R}^n: \sum_{i=1}^n u_i^2 = 1\}$
denotes the unit sphere. However, the vectors on these
spheres will get scaled by the multipliers $\lambda_i^{1/d}$
in (\ref{eq:Trank2}) when we pass to the fradeco
variety $\mathcal{T}_{r,n,d}$. To achieve better
geometric properties and computational speed,
we map each real sphere $\mathbb{S}^{n-1}$
to complex projective $(n{-}1)$-space  $\mathbb{P}^{n-1}$.

The {\em projective funtf variety} $\mathcal{G}_{r,n}$ is the image of
$\mathcal{F}_{r,n}$ in $(\mathbb{P}^{n-1})^r$.
To describe its equations,
we use an $n \times r$-matrix
$V = (v_{ij})$ of unknowns as before, but now the $i$-th column of
$V$ represents coordinates on the $i$-th
factor of $(\mathbb{P}^{n-1})^r$.
We introduce the $r\times r$ diagonal matrix 
\begin{equation}
\label{eq:Ddiag}
 D \quad = \quad {\rm diag} \biggl(
\sum_{i=1}^n v_{i1}^2\,,\,
\sum_{i=1}^n v_{i2}^2\, , \,
\ldots,\,\sum_{i=1}^n v_{ir}^2 \biggr). 
\end{equation}
The variety $\mathcal{G}_{r,n}$ is defined by the 
following matrix equation:
\begin{equation}
\label{eq:matrixeqn}
V \cdot D^{-1} \cdot V^T \,=\, \frac{r}{n} \cdot {\rm Id}_n .
\end{equation}
Each entry on the left hand side is a 
homogeneous rational function of degree $0$.
In fact, these functions are multihomogeneous:
they define rational functions on $(\mathbb{P}^{n-1})^r$.

The challenge is to clear denominators in (\ref{eq:matrixeqn}),
so as to obtain a system of polynomial equations that defines
$\mathcal{G}_{r,n}$ as a subvariety of $(\mathbb{P}^{n-1})^r$.
Next we solve this problem for $n=2$.

For planar frames, equation (\ref{eq:matrixeqn}) translates into the 
vanishing of the  two rational functions
\begin{equation}
\label{eq:PandQ}
P \,=\,\sum_{j=1}^r \frac{2v_{1j}^2 }{v_{1j}^2 + v_{2j}^2} \,- \,r
 \qquad \hbox{and} \qquad
Q \,=\,\sum_{j=1}^r \frac{2v_{1j} v_{2j} }{v_{1j}^2 + v_{2j}^2} .
\end{equation}
Consider the numerator of the rational function
$$ P  - i Q \quad = \quad
\sum_{j=1}^r \frac{v_{1j}^2 -2i v_{1j}v_{2j} - v_{2j}^2}{v_{1j}^2 + v_{2j}^2}
\quad = \quad \sum_{j=1}^r \frac{v_{1j} - v_{2j} i}{v_{1j} + v_{2j}i},
\qquad \hbox{where} \,\,\,i  = \sqrt{-1}.
$$
Let $\widetilde P$ and $\widetilde Q$ denote the real part and 
the imaginary part of that numerator.
These are two multilinear polynomials of degree $r$ with integer coefficients in
$v_{11},v_{12},\ldots,v_{2r}$. They define a complete intersection, and, by 
construction, this is precisely our funtf variety in $(\mathbb{P}^1)^r$:

\begin{lemma} \label{lem:multilinear}
The projective funtf variety $\mathcal{G}_{r,2}$ is a complete intersection of codimension~$2$ in
  $(\mathbb{P}^{1})^r$, namely, it is the zero set of the two
  multilinear forms $\widetilde P$ and $ \widetilde Q$.
\end{lemma}

Here are explicit formulas for the multilinear forms that define $\mathcal{G}_{r,2}$ when $r \leq 5$:

\begin{example} \rm
If $r = 3$, then
$\widetilde P = 3 v_{11} v_{12} v_{13} + v_{11} v_{22} v_{23} +  
v_{21} v_{12}v_{23} +  v_{21} v_{22} v_{13}$ and
$ \widetilde Q  =  v_{11} v_{12} v_{23} {+} v_{11}  v_{22} v_{13}
{+}  v_{21} v_{12} v_{13} {+} 3 v_{21} v_{22} v_{23} $.
 If $r {=} 4$, then $ \widetilde P  =  4 (v_{11} v_{12} v_{13} v_{14}
 -  v_{21} v_{22} v_{23} v_{24} ) $~and
 $$ \begin{matrix}
\widetilde Q & = & 2 v_{11} v_{12} v_{13} v_{24} 
+ 2 v_{11} v_{12} v_{23} v_{14} + 
2v_{11} v_{22}v_{13} v_{14}  + 2 v_{11} v_{22} v_{23} v_{24} +  \\ & & 
2  v_{21}v_{12} v_{13} v_{14} + 2 v_{21} v_{12} v_{23} v_{24} + 
2  v_{21} v_{22}v_{13} v_{24} + 2  v_{21} v_{22} v_{23} v_{14}. \,\,
\end{matrix}
$$
If $r = 5$, then
$$
\begin{matrix} 
\widetilde P  & = &5 v_{11} v_{12} v_{13} v_{14} v_{15} 
- v_{11} v_{12} v_{13} v_{24} v_{25} 
- v_{11} v_{12} v_{23} v_{14} v_{25} 
- v_{11} v_{12} v_{23} v_{24} v_{15}  \\ & & 
- v_{11} v_{22} v_{13} v_{14} v_{25} 
- v_{11} v_{22} v_{13} v_{24} v_{15} 
- v_{11} v_{22} v_{23} v_{14} v_{15}
- 3 v_{11} v_{22} v_{23} v_{24} v_{25} \\ & & 
-   v_{21} v_{12} v_{13} v_{14} v_{25} 
-   v_{21} v_{12} v_{13} v_{24} v_{15}
-   v_{21} v_{12} v_{23} v_{14} v_{15}
- 3 v_{21} v_{12} v_{23} v_{24} v_{25} \\ & & 
-   v_{21} v_{22} v_{13} v_{14} v_{15}
- 3 v_{21} v_{22} v_{13} v_{24} v_{25} 
- 3 v_{21} v_{22} v_{23} v_{14} v_{25} 
- 3 v_{21} v_{22} v_{23} v_{24} v_{15},
\end{matrix}
$$
and $\widetilde Q$ is obtained from $\widetilde P$ by
switching the two rows of $V$.
\hfill $\diamondsuit $
\end{example}

Such formulas are useful for parametrizing frames.
We write the equations for $\mathcal{G}_{r,2}$  as
$$ \begin{pmatrix} \widetilde P \\ \widetilde Q \end{pmatrix} \,\,  = \,\,
\begin{pmatrix} m_{11} & m_{12} \\ m_{21} & m_{22} \end{pmatrix} \cdot 
\begin{pmatrix} v_{1r} \\ v_{2r} \end{pmatrix} \,\, = \,\, \begin{pmatrix} 0 \\ 0 \end{pmatrix} .$$
The matrix entries $m_{ij}$ are multilinear forms in 
$(v_{11}:v_{21})$, $(v_{12}:v_{22})$, \ldots, $(v_{1,r-1}:v_{2,r-1})$.
Using the quadratic formula, we solve the following equation 
for one of its unknowns:
\begin{equation}
\label{eq:eliminant}
 m_{11} m_{22} \,=\, m_{12} m_{21}. 
 \end{equation}
This defines a hypersurface in  $(\mathbb{P}^1)^{r-1}$,
from which  we can now easily sample points.
The point in the remaining $r$th factor $\mathbb{P}^1$ is  then
recovered by setting
$\, v_{1r}  = m_{12} $, $\,v_{2r} = -m_{11}$.

For $n \geq 3$, we do not know the
generators of the multihomogeneous prime
ideal of $\mathcal{G}_{r,n}$. Here are two instances where
{\tt Macaulay2} \cite{M2} succeeded in computing these ideals:

\begin{example} \rm
The variety $\mathcal{G}_{4,3}$ is a threefold in $(\PP^2)^4$.
Its ideal is generated by $34$ quartics. Among them are
the equations that define the six coordinate projections into $(\PP^2)^2$, like
$$ \begin{matrix} 8 (v_{11}^2 v_{12}^2+ v_{21}^2 v_{22}^2+ v_{31}^2 v_{32}^2)
+18 (
  v_{11} v_{21} v_{12} v_{22}
+ v_{11} v_{31} v_{12} v_{32} 
+ v_{21} v_{31} v_{22} v_{32}) \\
-v_{11}^2 v_{22}^2
-v_{11}^2 v_{32}^2
-v_{21}^2 v_{12}^2
-v_{21}^2 v_{32}^2
-v_{31}^2 v_{12}^2 
-v_{31}^2 v_{22}^2.
\end{matrix} $$
\end{example}

\begin{example} \rm
Let $r=5$ and $n=3$. 
By saturating the denominators in (\ref{eq:matrixeqn}),
we found that the ideal of $\mathcal{G}_{5,3}$ is
generated by  a $120$-dimensional  ${\rm SO}_3$-invariant
space of sextics. The following polynomial (with $60$ terms
of $\mathbb{Z}^5$-degree $(2,2,2,0,0)$) is a highest weight vector:
\scriptsize
$$ \begin{matrix}
50 v_{11}^2 v_{12}^2 v_{13}^2
+5 v_{11}^2 v_{12}^2 v_{23}^2
+5 v_{11}^2 v_{12}^2 v_{33}^2
+45 v_{11}^2 v_{12} v_{22} v_{13} v_{23}
+45 v_{11}^2 v_{12} v_{32} v_{13} v_{33} 
+5 v_{11}^2 v_{22}^2 v_{13}^2
+5 v_{11}^2 v_{22}^2 v_{23}^2
-4 v_{11}^2 v_{22}^2 v_{33}^2  \\
+18 v_{11}^2 v_{22} v_{32} v_{23} v_{33}
+5 v_{11}^2 v_{32}^2 v_{13}^2
-4 v_{11}^2 v_{32}^2 v_{23}^2
+5 v_{11}^2 v_{32}^2 v_{33}^2
+45 v_{11} v_{21} v_{12}^2 v_{13} v_{23}
+45 v_{11} v_{21} v_{12} v_{22} v_{13}^2 
+18 v_{11} v_{21} v_{32}^2 v_{13} v_{23} 
\\
+45 v_{11} v_{21} v_{12} v_{22} v_{23}^2 
+18 v_{11} v_{21} v_{12} v_{22} v_{33}^2
+27 v_{11} v_{21} v_{12} v_{32} v_{23} v_{33}
+45 v_{11} v_{21} v_{22}^2 v_{13} v_{23}
+27 v_{11} v_{21} v_{22} v_{32} v_{13} v_{33} \\
+45 v_{11} v_{31} v_{12}^2 v_{13} v_{33}  
+27 v_{11} v_{31} v_{12} v_{22} v_{23} v_{33} 
+45 v_{11} v_{31} v_{12} v_{32} v_{13}^2
+18 v_{11} v_{31} v_{12} v_{32} v_{23}^2
+45 v_{11} v_{31} v_{12} v_{32} v_{33}^2 
-4 v_{21}^2 v_{12}^2 v_{33}^2
\\
+18 v_{11} v_{31} v_{22}^2 v_{13} v_{33} 
+27 v_{11} v_{31} v_{22} v_{32} v_{13} v_{23} 
+45 v_{11} v_{31} v_{32}^2 v_{13} v_{33} 
+5 v_{21}^2 v_{12}^2 v_{13}^2 
+5 v_{21}^2 v_{12}^2 v_{23}^2
+45 v_{21}^2 v_{12} v_{22} v_{13} v_{23} \\ 
+18 v_{21}^2 v_{12} v_{32} v_{13} v_{33} 
+5 v_{21}^2 v_{22}^2 v_{13}^2 
+50 v_{21}^2 v_{22}^2 v_{23}^2 
+5 v_{21}^2 v_{22}^2 v_{33}^2 
+45 v_{21}^2 v_{22} v_{32} v_{23} v_{33}
-4 v_{21}^2 v_{32}^2 v_{13}^2
+5 v_{21}^2 v_{32}^2 v_{23}^2
+5 v_{21}^2 v_{32}^2 v_{33}^2 \\
+18 v_{21} v_{31} v_{12}^2 v_{23} v_{33} 
+27 v_{21} v_{31} v_{12} v_{22} v_{13} v_{33} 
+27 v_{21} v_{31} v_{12} v_{32} v_{13} v_{23} 
+45 v_{21} v_{31} v_{22}^2 v_{23} v_{33}
+18 v_{21} v_{31} v_{22} v_{32} v_{13}^2 \\
+45 v_{21} v_{31} v_{22} v_{32} v_{23}^2 
+45 v_{21} v_{31} v_{22} v_{32} v_{33}^2  
+45 v_{21} v_{31} v_{32}^2 v_{23} v_{33} 
+5 v_{31}^2 v_{12}^2 v_{13}^2 
-4 v_{31}^2 v_{12}^2 v_{23}^2
+5 v_{31}^2 v_{12}^2 v_{33}^2
+18 v_{31}^2 v_{12} v_{22} v_{13} v_{23} \\
+45 v_{31}^2 v_{12} v_{32} v_{13} v_{33} 
-4 v_{31}^2 v_{22}^2 v_{13}^2
+5 v_{31}^2 v_{22}^2 v_{23}^2 
+5 v_{31}^2 v_{22}^2 v_{33}^2
+45 v_{31}^2 v_{22} v_{32} v_{23} v_{33}
+5 v_{31}^2 v_{32}^2 v_{13}^2
+5 v_{31}^2 v_{32}^2 v_{23}^2
+50 v_{31}^2 v_{32}^2 v_{33}^2.
\end{matrix}
$$
\normalsize
The ideal of $\mathcal{G}_{5,3} \subset (\PP^2)^5$ has
 $10$ generators like this, each spanning
a  one-dimensional graded component.
It has $30$ components
of degrees like $(2,2,1,1,0)$, each 
generated by a polynomial with $78$ terms.
Finally, it has five $16$-dimensional
components of degrees like $(2,1,1,1,1)$.
\hfill $\diamondsuit$
\end{example}
\smallskip

In order to sample points from the funtf variety $\mathcal{F}_{r,n}$,
we can also use  the following parametrization found in \cite{CStr, Str}.
We write $V = (U', W)$, where $U'$ is
an $n \times n$-matrix and $W$ is an $(r-n) \times n$-matrix.
For the columns of $W $ we take
arbitrary points on the unit sphere $\mathbb{S}^{n-1}$.
In practice, it is convenient to fix a rational parametrization
of $\mathbb{S}^{n-1}$, so as to ensure that $W$ has
rational entries $w_{ij}$. For instance, for $n=3$ we might use
the following formulas:
\begin{equation}
\label{eq:strawpara1}
\!\! w_{1j} =  \frac{2 \lambda_j \mu_j}{ \lambda_j^2 + \mu_j^2 + \nu_j^2} ,\,\,
w_{2j} = \frac{2 \lambda_j \nu_j}{ \lambda_j^2 + \mu_j^2 + \nu_j^2},\,\,
w_{3j} = \frac{\lambda_j^2 - \mu_j^2 - \nu_j^2}{ \lambda_j^2 + \mu_j^2 + \nu_j^2},
\quad {\rm where} \,\,\lambda_j,\mu_j,\nu_j \in \mathbb{Z}. \
\end{equation}
After these choices have been made, we fix the
following $n \times n$-matrix with entries in $\mathbb{Q}$:
\begin{equation}
\label{eq:strawpara2}
 S \,\, = \,\, \frac{r}{n} \cdot {\rm Id}_n \,-\, W \cdot W^{T}. 
 \end{equation}
 It now remains to study all $n \times n$-matrices $U = (u_{ij})$
 that satisfy
 $$
U \cdot D^{-1} \cdot U^{T} \,=\, S ,
\,\,\hbox{where} \,\,\,\,
 D \,\, = \,\, {\rm diag} \bigl(\,
\sum_{i=1}^n u_{i1}^2\,,\,
\ldots,\,\sum_{i=1}^n u_{in}^2\, \bigr). 
$$
For any such $U$ we get a funtf $V  = (U',W) \in \mathcal{F}_{r,n}$ by setting $U' = U\cdot D^{-1/2}$.
For random choices in (\ref{eq:strawpara1}), the matrix
 $S$ is invertible, and  the previous equation is equivalent to
\begin{equation}
\label{eq:strawpara3}
D \,\,=\,\, U^{T} \cdot S^{-1} \cdot U.
\end{equation}
This identity of symmetric matrices
defines $\binom{n+1}{2}$ equations
in the entries $u_{ij}$ of $U$. 
The equation in position $(i,j)$ is bilinear in
 $(u_{1i}, u_{2i},\ldots,u_{ni})$ and
 $(u_{1j}, u_{2j},\ldots,u_{nj})$.
We solve the system (\ref{eq:strawpara3}) iteratively for the columns of $U$.
 We begin with the (1,1)  entry of (\ref{eq:strawpara3}).
 There are $n-1$ degrees of freedom to fill in the first column of $U$,
 then $n-2$ degrees of freedom to fill in the second column, etc.
This involves repeatedly solving quadratic equations in
one variable, so each solution lives in a tower
of quadratic extensions over $\mathbb{Q}$. In summary:

\begin{proposition}
The equations (\ref{eq:strawpara1}),
(\ref{eq:strawpara2}), (\ref{eq:strawpara3})
represent a parametrization of $\mathcal{F}_{r,n}$.
\end{proposition}

The rotation group ${\rm SO}_n$ acts by left multiplication
on the funtf variety $\mathcal{F}_{r,n}$. There is 
a natural way to construct the quotient
$\mathcal{F}_{r,n}/{\rm SO}_n$ as an algebraic variety,
namely by mapping it into the Grassmannian
${\rm Gr}(n,r)$ of $n$-dimensional subspaces
of $\CC^r$. This is described by Cahill and Strawn in \cite[Section 3.1]{CStr},
and we briefly develop some basic algebraic properties.

We here define ${\rm Gr}(n,r)$ to be the image of the {\em Pl\"ucker map}
$ \CC^{n \times r} \rightarrow \CC^{\binom{r}{n}}$
that takes an $n \times r$-matrix $V $ to its vector $p = p(V)$
of $n \times n$-minors. The coordinates
$p_I$ of $p$ are indexed by the set $\binom{[r]}{n}$ of $n$-element
subsets of $[r] = \{1,2,\ldots,r\}$. With this definition, ${\rm Gr}(n,r)$ is the
affine subvariety of
$\CC^{\binom{r}{n}}$ defined by the {\em quadratic Pl\"ucker relations},
such as $p_{12} p_{34} - p_{13} p_{24} + p_{14} p_{23} = 0$
for $n=2,r=4$.
The dimension of ${\rm Gr}(n,r) $ is  $(r-n)n + 1$.
Note that if $V V^T = (r/n) \cdot {\rm Id}_n$, then
the {\em Cauchy-Binet formula}  (cf.~\cite[Prop.~6]{CStr}) implies
\begin{equation}
\label{eq:sum1}
\sum_{I \in \binom{[r]}{n}} p_I^2 \,\, = \,\, \left( \frac{r}{n} \right)^n. 
\end{equation}
The real points in ${\rm Gr}(n,r)$, up to scaling, correspond to 
$n$-dimensional subspaces of $\RR^r$.

\begin{proposition}
The image of $\mathcal{F}_{r,n}$ under the Pl\"ucker map
is an affine variety of dimension  
$(r{-}n)n - r + 2$ in the Grassmannian ${\rm Gr}(n,r) \subset \CC^{\binom{r}{n}}$. It is defined by the 
equations
\begin{equation}
\label{eq:sum2}
 \sum_{I: i \in I} p_I^2 \,\,\,=\,\,\, \left(\frac{r}{n}\right)^{n-1}
\qquad \hbox{for} \,\,\, i = 1,2,\ldots,r . 
\end{equation}
The real points in this image correspond to ${\rm SO}_n$-orbits
of $n$-dimensional frames in $\mathcal{F}_{r,n}$.
\end{proposition}

Note that adding up the $r$ relations in
(\ref{eq:sum2}) and dividing by $n$ gives precisely
(\ref{eq:sum1}).

\begin{proof}
Both $\mathcal{F}_{r,n}$ and the constraints (\ref{eq:sum2})
are invariant under ${\rm SO}_n$.
Suppose that $V \in \CC^{n \times r}$ satisfies $V V^T = (r/n) \cdot {\rm Id}_n$.
We may assume (modulo ${\rm SO}_n$) that
 the $i$-th column of $V$
is $(\alpha,0,\ldots,0)^T$ for some
$\alpha \in \CC$. Let $\tilde V$ be the matrix obtained from 
$V$ by deleting the first row and $i$-th column.
Then $\,\tilde V \cdot {\tilde V}^T \,=\,(r/n) \cdot {\rm Id}_{n-1}$.
Any $p_{I}$ with $i \in I$ equals
$\alpha$ times the maximal minor 
of $\tilde V$ indexed by $I \backslash \{i\}$.
Applying (\ref{eq:sum1}) to $\tilde V$, this gives
$$  \sum_{I: i \in I} p_I^2 \,\,\, = \,\,\, \alpha^2 \cdot \left(\frac{r}{n}\right)^{n-1}. $$
Hence (\ref{eq:sum2}) holds if and only if $\alpha = \pm 1$, and
this holds for all $i$ if and only if $V$ lies in $\mathcal{F}_{r,n}$.
The dimension formula follows from Theorem~\ref{thm:strawn}
because ${\rm SO}_n$ acts faithfully on $\mathcal{F}_{r,n}$.
\end{proof}

\begin{example} \rm
Let $n=2$. If $r=5$, then our construction realizes 
$\mathcal{F}_{5,2}/{\rm SO}_2$ as an irreducible
surface of degree $80$ in $\CC^{10}$. Its prime ideal is generated by the
ten quadratic polynomials
$$ \begin{matrix}
 p_{14} p_{23}-p_{13} p_{24}+p_{12} p_{34}, \,
 p_{15} p_{23}-p_{13} p_{25}+p_{12} p_{35}, \,
 p_{15} p_{24}-p_{14} p_{25}+p_{12} p_{45}, \,
 p_{15} p_{34}-p_{14} p_{35} \\ +p_{13} p_{45}, \, 
 p_{25} p_{34}-p_{24} p_{35}+p_{23} p_{45}, \,
     p_{12}^2+p_{13}^2+p_{14}^2+p_{15}^2-5/2,\,
     p_{12}^2 + p_{23}^2 + p_{24}^2 + p_{25}^2 - 5/2, \,  \\
   p_{13}^2 + p_{23}^2 + p_{34}^2 + p_{35}^2 - 5/2, \, 
   p_{14}^2 + p_{24}^2 + p_{34}^2 + p_{45}^2 - 5/2, \, 
   p_{15}^2 + p_{25}^2 + p_{35}^2 + p_{45}^2 - 5/2.
     \end{matrix}
     $$
     If $r=4$, then $\mathcal{F}_{4,2}/{\rm SO}_2$ 
is a reducible curve of degree $24$ in
     $\CC^6$. Its defining equations are
     $$ p_{14} p_{23}-p_{13} p_{24}+p_{12} p_{34} = 0 ,\,\,
 p_{12}^2+p_{13}^2+p_{14}^2 =  
     p_{12}^2 + p_{23}^2 + p_{24}^2  = 
   p_{13}^2 + p_{23}^2 + p_{34}^2  = 
      p_{14}^2 + p_{24}^2 + p_{34}^2  = 2. $$
      As in Example \ref{eq:vierzwei},
this curve breaks into $12$  components.
One of these $12$ irreducible curves is
$\,
\bigl\{ \,p \in \CC^6 \,:\,
p_{12} = p_{34} = 1 ,\,
p_{13} = p_{24},\,
p_{14} = - p_{23},\,
p_{23}^2 + p_{24}^2 = 1 \bigr\}$.
\hfill $\diamondsuit$
\end{example}

The analogous decomposition is found easily for the case $r=n+1$.
Here, there are no Pl\"ucker relations, so ${\rm Gr}(n,n+1) \simeq \mathbb{S}^{n}$. 
For convenience of notation,  we set $q_i = p_{[n+1]\backslash \{i\}}$
 in (\ref{eq:sum2}). The quotient
space $\mathcal{F}_{n+1,n} /{\rm SO}_n$ is the
subvariety of $\CC^{n+1}$ defined by the equations
$$ q_1^2 + q_2^2 + \cdots + q_n^2 + q_{n+1}^2 \,\, = \,\,
(n+1)^{n-1}/n^n
+  \,\, q_i^2
\qquad \hbox{for} \,\, i = 1,2,\ldots,n+1. $$
These are equivalent to the following equations, which imply
Corollary~\ref{cor:quotientspace}:
$$ q_1^2 \,=\, q_2^2 \,= \, q_3^2 \,=\, \cdots \,=\, q_{n+1}^2 \,\, = \,\, (n+1)^{n-1}/n^{n+1}. $$

\begin{corollary} \label{cor:quotientspace}
The quotient space $\mathcal{F}_{n+1,n}/{\rm SO}_n$ is a variety consisting of
 $\,2^{n+1}$ isolated points in $\RR^{n+1} = {\rm Gr}(n,n+1)$, namely those
 points with coordinates  $\pm (n+1)^{(n-1)/2}/n^{(n+1)/2}$.
\end{corollary}

Any of the $2^{n+1}$ components of $\mathcal{F}_{n+1,n}$
can be used to parametrize our variety $\mathcal{T}_{n+1,n,d}$.

\begin{example} \rm
Let $n=3$. The point $p = \sqrt{3}(\frac{4}{9},\frac{4}{9},\frac{4}{9},\frac{4}{9})$
in ${\rm Gr}(3,4) $ corresponds to the ${\rm SO}_3$-orbit
of the frame $V$ in Example \ref{ex:opening}.
The variety $\mathcal{G}_{4,3}$ can be parametrized as follows:
$$
V = (v_{ij}) \,= \,
\begin{bmatrix}
1-2 y^2-2 z^2 &  2 x y-2 z w & 2xz+2yw \\
2 x y+2 z w & \! \! 1-2 x^2-2 z^2  \! & 2 y z-2 x w \\
2 x z-2 y w & 2 yz+2xw & \! 1-2x^2-2y^2 
\end{bmatrix} \! \!
\begin{bmatrix}
 3 &  1 &   1 &  \!\! -5 \\
 3 & 1 & \!\! -5 &   1 \\
 3 & \!\! -5 &   1 &   1
 \end{bmatrix} \!\!
 \begin{bmatrix}
 \nu_1 & 0 & 0 & 0 \\
 0 & \nu_2 & 0 & 0 \\
 0 & 0 & \nu_3 & 0 \\
 0 & 0 & 0 & \nu_4
 \end{bmatrix} \!\! . 
$$
The $3 \times 3$-matrix on the left is the familiar
parametrization of ${\rm SO}_3$ via unit quaternions.
This gives the parametrization of the fradeco variety 
$\mathcal{T}_{4,3,d} $ seen in (\ref{eq:ex434c}).
\hfill $\diamondsuit$
\end{example}

The embedding of $\mathcal{F}_{r,n}/{\rm SO}_n$ into  ${\rm Gr}(r,n)$ via (\ref{eq:sum2})
connects  frame theory with {\em matroid theory}.
The matroid of $V$ is given by the set of Pl\"ucker coordinates
$p_I$ that are zero. If all Pl\"ucker coordinates
are nonzero, then the matroid is uniform.
It is a natural to ask which matroids 
are realizable over $\RR$ when the additional constraints 
(\ref{eq:sum2}) are imposed. 

The discussion in \cite[Section 3.2]{CStr}
relates frame theory to the study of {\em orbitopes} \cite{SSS}.
Cahill and Strawn set up an optimization problem for computing
Parseval frames that are most uniform. Their formulation in \cite[p.~24]{CStr} 
is a linear program over the {\em Grassmann orbitope}, which is
the convex hull of ${\rm Gr}(n,r)$ intersected with (\ref{eq:sum1}).
The same optimization problem makes sense with ${\rm Gr}(n,r)$ replaced by 
$\mathcal{F}_{r,n}/{\rm SO}_n$,
or, algebraically, with (\ref{eq:sum1}) replaced by (\ref{eq:sum2}). 
If $n=2$, then the former problem is a {\em semidefinite program}.
This is the content of \cite[Theorem 7.3]{SSS}. For $n \geq 3$, the situation 
is more complicated, but the considerable body of results coming from
calibrated manifolds, such as \cite[Theorem 7.5]{SSS}, should still be helpful.

\section{Binary forms}
\label{sec3}

We now commence our study of the fradeco variety $\mathcal{T}_{r,n,d}$.
In this section we focus on the case $n=2$ of binary forms
that are decomposable into small frames.
The case $r=2$ is the odeco surface known from \cite[\S 3]{Rob}.
Proposition 3.6 in \cite{Rob} gives an explicit list
of quadrics that forms a Gr\"obner basis for the prime ideal
of $\mathcal{T}_{2,2,d}$,
and these are here expressed as the
$2 \times 2$-minors of a certain $3 \times (d-3)$-matrix $\mathcal{M}_{4}$.
What follows is our main result in Section~\ref{sec3}.
We are using coordinates $(t_0:\cdots:t_d)$ for
the space $\PP^d  = \PP({\rm Sym}_d (\CC^2))$ of binary forms.
In the notation of (\ref{eq:generalp}), the coordinate
$t_i$ would be $t_{111\cdots1222\cdots2}$
with $i$ indices $1$ and $d-i$ indices~$2$.

\begin{theorem}
\label{thm:sec3main}
Fix $r \in \{3,4,\ldots,9\}$ and $d\geq 2r-2$. There exists a matrix
$\mathcal{M}_{r}$ such that:
\begin{itemize}
\vspace{-0.05in}
\item[(a)] Its maximal minors form a
Gr\"obner basis for the prime ideal of $\,\mathcal{T}_{r,2,d}$. \vspace{-0.1in}
\item[(b)] It has $r-1$ rows and $d-r+1$ columns, and the entries are linear
forms in $t_0,\ldots,t_d$. \vspace{-0.1in}
\item[(c)] Each column involves $r$ of the unknowns $t_i$, and they
are identical up to index shifts.
\end{itemize}
These matrices can be chosen as follows:
\begin{equation}
\label{eq:M_{3}}
\mathcal{M}_{3} \quad  = \quad
\begin{pmatrix}
t_0-3t_2 & t_1-3t_3 & t_2-3t_4 &  t_3 - 3 t_5 & \cdots & t_{d-3} - 3 t_{d-1} \\
3 t_1-t_3 & 3 t_2-t_4 & 3t_3-t_5 & 3 t_4 - t_6 & \cdots &  3t_{d-2}- t_d \\
\end{pmatrix}
\end{equation}
\begin{equation}
\label{eq:M_{4}}
\mathcal{M}_{4} \quad  = \quad
\begin{pmatrix}
\,t_0 + t_4 \,&  \,t_1 + t_5\, & \, t_2 + t_6 \,&\, t_3 + t_7\, & \cdots &\, t_{d-4} + t_d \, \\
\,t_1 - t_3  \,& \, t_2 - t_4 \, & \,t_3 - t_5 \,&\,   t_4 - t_6 \,& \cdots  & \,t_{d-3} + t_{d-1} \, \\
     t_2     &      t_3      &      t_4     &  t_5  & \cdots  &  t_{d-2} 
\end{pmatrix}
\end{equation}
\begin{equation}
\label{eq:M_{5}}
\mathcal{M}_{5} \quad  = \quad 
\begin{pmatrix}
t_0 + 5 t_2 & t_1 + 5 t_3 & t_2 + 5 t_4 & t_3 + 5 t_5 & \cdots & t_{d-5} + 5 t_{d-3} \\
t_1 - 3 t_3 & t_2 - 3 t_4   & t_3 - 3 t_5 & t_4 - 3 t_6 &  \cdots & t_{d-4} - 3 t_{d-2} \\
3 t_2 - t_4 & 3 t_3 - t_5   & 3 t_4 - t_6 & 3 t_5 - t_7 & \cdots  & 3 t_{d-3} - t_{d-1} \\
5 t_3 + t_5 &  5 t_4 + t_6  &  5 t_5 + t_7 &  5 t_6 + t_8 &  \cdots & 5 t_{d-2} + t_d
\end{pmatrix}
\end{equation}
\begin{equation}
\label{eq:M_{6}}
\mathcal{M}_{6} \quad  = \quad
\begin{pmatrix}  
t_0 + 3 t_2 & t_1 + 3 t_3 & t_2 + 3 t_4 & t_3 + 3 t_5 & \cdots & t_{d-6} + 3 t_{d-4} \\
 t_1 + t_5   &  t_2 + t_6   &  t_3 + t_7   &  t_4 + t_8     & \cdots &  t_{d-5} + t_{d-1}   \\
 t_2  - t_4   &  t_3  - t_5   &  t_4  - t_6   &  t_5  - t_7    &  \cdots &  t_{d-4} - t_{d-2}  \\
    t_3         &       t_4       &       t_5        &       t_6       &      \cdots   &  t_{d-3} \\
    3t_4 + t_6   &  3t_5 +  t_7   &  3t_6 +  t_8   &  3t_7 +  t_9    & \cdots &  3t_{d-2} +  t_d   \\
\end{pmatrix}
\end{equation}
The first column of $\mathcal{M}_7$ is
$(3 t_0 + 7 t_2, t_1 + 5 t_3, t_2 - 3 t_4, 3 t_3 - t_5, 5 t_4 + t_6, 7 t_5 + 3 t_7)^T$, 
the first column of $\mathcal{M}_8$ is
$(t_0 + 2 t_2, t_1 + 3 t_3, \,t_4,\, t_3 - t_5, t_2 + t_6, 3 t_5 + t_7, 2 t_6 + t_8)^T$,
and the first column of $\mathcal{M}_9$ is
$(5 t_0+9t_2, 3 t_1 + 7 t_3, t_2 + 5 t_4 , t_3 - 3 t_5, 3 t_4-t_6, 5 t_5 + t_7, 7 t_6+3t_8,
9t_7+5t_9)^T$.
\end{theorem}

We conjecture that the same result holds for all $r$,
and we explain what we currently know after the proof.
Let us begin with a lemma concerning the dimension of our variety.

\begin{lemma} \label{lem:fradim}
The fradeco variety $\mathcal{T}_{r,2,d}$ is
irreducible and has dimension ${\rm min}(2r-3,d)$.
\end{lemma}

\begin{proof}
For $d \geq 5$, the funtf variety $\mathcal{F}_{r,2} \subset  (\mathbb{S}^1)^r$
is irreducible, by Theorem~\ref{thm:strawn},
and hence so is its closure $\mathcal{G}_{r,2}$ in $(\mathbb{P}^1)^r$.
While the two special varieties $\mathcal{F}_{3,2}$
and $\mathcal{F}_{4,2}$ are reducible, the analyses in
Example \ref{eq:vierzwei} and Corollary \ref{cor:quotientspace}
show that $\mathcal{G}_{3,2}$
and $\mathcal{G}_{4,2}$ are irreducible.

Regarding $\mathcal{G}_{r,2}$ as an affine variety in $\CC^{2\times r}$,
we obtain $\mathcal{T}_{r,2,d}$ as its image under the map
\begin{equation}
\label{eq:para_n=2}  t_i \,\,= \,\,
v_{11}^i v_{21}^{d-i} + 
v_{12}^i v_{22}^{d-i} + 
v_{13}^i v_{23}^{d-i} + 
\cdots + 
v_{1r}^i v_{2r}^{d-i} \qquad \hbox{for}\,\,\, i = 0,1,\ldots,d.
\end{equation}
This proves that $\mathcal{T}_{r,2,d}$ is irreducible.
To see that it has the expected dimension, consider the 
$r$-th secant variety of the rational normal curve in $\mathbb{P}^d$, which 
is the image of the map $\CC^{2 \times r} \dashrightarrow
\mathbb{P}^d$ given by (\ref{eq:para_n=2}).
It is known that this secant variety has the expected dimension,
namely ${\rm min}(2r-1,d)$, and the fiber dimension of the map (\ref{eq:para_n=2}) 
does not jump unless some $2 \times 2$-minor of $V = (v_{ij})$ is zero.
Since ${\rm codim}(\mathcal{G}_{r,2}) = 2$, by Lemma~\ref{lem:multilinear},
the claim follows.
\end{proof}

\begin{proof}[Proof of Theorem \ref{thm:sec3main}] We first show that the maximal minors of 
our matrices $\mathcal M_r$ vanish on the fradeco variety $\mathcal{T}_{r,2,d}$
for $r = 3,4,\ldots,9$. After substituting the parametrization (\ref{eq:para_n=2}) for
$t_0,t_1,\ldots,t_d$, we can decompose these matrices as follows:
$$
\mathcal M_r \,\,=\,\, M_r \cdot \begin{pmatrix}
v_{11}^{d-r} & v_{11}^{d-r-1}v_{21} & v_{11}^{d-r-2} v_{21}^2 & \cdots & v_{21}^{d-r}\\
v_{12}^{d-r} & v_{12}^{d-r-1}v_{22} & v_{12}^{d-r-2} v_{22}^2 & \cdots & v_{22}^{d-r}\\
 \vdots & \vdots & \vdots &  \ddots & \vdots \\
v_{1r}^{d-r} & v_{1r}^{d-r-1}v_{2r} & v_{1r}^{d-r-2} v_{2r}^2 & \cdots & v_{2r}^{d-r}
\end{pmatrix},
$$
where
$$ M_3 \,\,=\,\, \begin{pmatrix} 
(v_{22}^2-3v_{11}^2)v_{21} & (v_{22}^2 - 3v_{12}^2) v_{22} & (v_{23}^2 - 3v_{13}^2) v_{23} \\
(3v_{21}^2 - v_{11}^2)v_{11} & (3v_{22}^2 - v_{12}^2)v_{12} & (3v_{23}^2 - v_{13}^2)v_{13}
\end{pmatrix},
$$
$$
M_4 = \begin{pmatrix} 
v_{21}^4 + v_{11}^4 & v_{22}^4 + v_{12}^4 & v_{23}^4 + v_{13}^4 & v_{24}^4 + v_{14}^4 \\
v_{11}v_{21}^3 - v_{11}^3 v_{21} & v_{12}v_{22}^3 - v_{12}^3 v_{22} & 
v_{13}v_{23}^3 - v_{13}^3 v_{23} & v_{14}v_{24}^3 - v_{14}^3 v_{24}\\
v_{11}^2v_{21}^2 & v_{12}^2v_{22}^2 & v_{13}^2v_{23}^2 & v_{14}^2v_{24}^2
\end{pmatrix},
$$
$$
\begin{small}
M_5 = \begin{pmatrix} v_{21}^5 + 5v_{11}^5 & v_{22}^5 
+ 5v_{12}^5 & v_{23}^5 + 5v_{13}^5 & v_{24}^5 + 5v_{14}^5 & v_{25}^5 + 5v_{15}^5\\
v_{11}v_{21}^4 - 3v_{11}^3v_{21}^2 & v_{12}v_{22}^4 - 3v_{12}^3v_{22}^2 & v_{13}v_{23}^4 - 
3v_{13}^3v_{23}^2 & v_{14}v_{24}^4 - 3v_{14}^3v_{24}^2 & v_{15}v_{25}^4 - 3v_{15}^3v_{25}^2\\
3v_{11}^2v_{21}^3 - v_{11}^4v_{21} & 3v_{12}^2v_{22}^3 - v_{12}^4v_{22} & 3v_{13}^2v_{23}^3 - v_{13}^4v_{23} & 3v_{14}^2v_{24}^3 - v_{14}^4v_{24} & 3v_{15}^2v_{25}^3 - v_{15}^4v_{25} \\
5v_{11}^3v_{21}^2 + v_{11}^5 & 5v_{12}^3v_{22}^2 + v_{12}^5 & 5v_{13}^3v_{23}^2 
+ v_{13}^5 & 5v_{14}^3 v_{24}^2 + v_{14}^5 & 5v_{15}^3v_{25}^2 + v_{15}^5
\end{pmatrix}\!,
\end{small}
$$
and similarly for $M_6,M_7, M_8$ and $M_9$. We claim that 
the matrices $M_r$ have rank $<r-1$ whenever $V\in\mathcal F_{r, 2}$.
Equivalently, the $(r-1)\times (r-1)$ minors of $M_r$ lie in the ideal of $\mathcal G_{r, 2}$. 
It suffices to consider the leftmost such minor since 
all minors are equivalent under permuting the columns of $V$.
For each $r \leq 9$, we check that the determinant of that minor factors as
\begin{equation}
\label{eq:amazingfact}
 (m_{11} m_{22} - m_{12} m_{21}) \,\,\cdot \!\!
\prod_{1 \leq i < j \leq r-1}  (v_{1 i} v_{2j} - v_{2i} v_{1j} )  ,
\end{equation}
where the left factor is the polynomial of degree $2r-2$ given in (\ref{eq:eliminant}).
That polynomial vanishes on  $\mathcal{G}_{r,2}$.
This implies ${\rm rank}(M_r) \leq  r{-}2$ on $\mathcal{G}_{r,2}$,
and hence ${\rm rank}(\mathcal{M}_r) \leq r{-}2$ on $\mathcal{T}_{r,2,d}$.

Fix the lexicographic term order on $\CC[t_0,t_1,\ldots,t_d]$.
We can check that, for each $r\in\{3, 4, ..., 9\}$, the leading monomial of the 
leftmost maximal minor of $\mathcal{M}_r$ equals $t_0 t_2 t_4 \cdots t_{r-2}$.
Hence all $\binom{d-r+1}{r-1}$ maximal minors of $\mathcal{M}_r$ are
squarefree, and they generate the ideal
$$ I_{r,d}\,\,= \,\, \bigl\langle \,
t_{i_1} t_{i_2} t_{i_3} \cdots t_{i_{r-1}} \,\,:\,\,
2 \leq i_1 {+} 2 \leq i_2, \,i_2 {+} 2 \leq i_3,\,
i_3 {+} 2 \leq t_4,\, \ldots, \,i_{r-2} {+} 2 \leq i_{r-1} \leq d-2 \,
\bigr\rangle.
$$
This squarefree monomial ideal is pure of codimension $d-2r+3$ and 
it has degree  $\binom{d-r+1}{r-2} $.
This follows from \cite[Theorem 1.6]{Murai}. Indeed, in Murai's theory,
our ideal $I_{r,d}$ is obtained from the power of the maximal ideal by
applying the stable operator given by $a = (2,4,6,\ldots)$.

Combinatorial analysis reveals that the ideal $I_{r,d}$ is the intersection of the prime ideals
$$ \bigl\langle
t_{j_0} , t_{j_1}, t_{j_2}, t_{j_3} , \ldots, t_{j_{d-2r+2}} 
\bigr\rangle,
$$
where $\,j_0,j_2,j_4, \ldots $ are even, 
$\,j_1,j_3, j_5, \ldots $ are odd, and
$0 \leq j_0 {<} j_1 {<} j_2 {<} \cdots {<} j_{d-2r+2} \leq d$.
Note that number of such sequences is
$\binom{d-r+1}{d-2r+3}  = \binom{d-r+1}{r-2} $. Hence
the codimension and degree of $I_{r,d}$ are 
as expected for the ideal of maximal minors
of an $(r-1) \times (d-r+1)$-matrix with linear entries \cite[Ex.~19.10]{H}. 
The monomial ideal $I_{r,d}$ is Cohen-Macaulay because its corresponding
simplicial complex is shellable
(cf.~\cite[\S III.2]{Stanley}). Indeed, if we list
the associated primes in a dictionary order for all
sequences $\, j_0 j_1 j_2   \cdots j_{d-2r+2} \,$ as above,
then this gives a shelling order.

Using Buchberger's S-pair criterion, we check that the
maximal minors of $\mathcal{M}_r$ form a Gr\"obner basis.
We only need to consider pairs of minors whose leading terms
share variables. Up to symmetry, there are only few such pairs, so this is an easy check for 
each fixed $r \leq 9$.

Since $I_{r,d}$ is radical of codimension $d-2r+3$, 
we conclude that the ideal of maximal minors of $\mathcal M_r$
is radical and has the same codimension. However, that
ideal of minors is contained in the prime ideal of $\mathcal{T}_{r,2,d}$,
which has codimension $d-2r+3$ by Lemma \ref{lem:fradim}.

Therefore, we now know that $\mathcal{T}_{r,2,d}$ is
one of the irreducible components of the variety of
maximal minors of $\mathcal{M}_r$. To conclude the proof
we need to show that the latter variety is irreducible, so they are equal.
To see this, we fix $r$ and we proceed by induction on $d$.
For $d = 2r-2$, when $\mathcal{M}_r$ is a square matrix, 
this can be checked directly. To pass from $d$ to $d+1$,
we factor the matrix as $M_r$ times the rank $r$
Hankel matrix associated with a funtf $V$.
Increasing the value of $d$ to $d+1$ multiplies the $i$-th row
of the Hankel matrix by  $v_{i1}$ and it adds one more column.
This gives us the value for the new variable $t_{d+1}$.
Now, since that variable occurs linearly in the
maximal minors, its value is unique. This implies that
the unique rank $r-2$ extension from the old to the new $\mathcal{M}_r$
must come from the funtf $V$.
\end{proof}

We established Theorem \ref{thm:sec3main} assuming that $r \leq 9$,
but we believe that it holds for all $r$:

\begin{conjecture} For all $r \geq 3$ there exists a matrix 
$\mathcal{M}_{r}$ which satisfies the properties (a),(b) and (c) in Theorem \ref{thm:sec3main}.
When $r$ is odd, the matrix $\mathcal M_r$ is given by
\tiny
\begin{align*}
 \begin{pmatrix} r{-}4 & & & &&&&\\
& \!\!\! r{-}6 \!\! & & &&&&\\
&&\! \! \ddots\!&&&&&\\
& & & \! \!\! \!\! -1 \!\!\! & &&&\\
& &&&\! 3 \!&&& \\
&&&&& \! 5\! &&\\
&&&&&&\!\!\! \ddots & \\
&&&&&&\!\! &\!\!\! r\,
\end{pmatrix}\begin{pmatrix}t_0 & t_1 & \cdots & t_{d-r}\\
t_1 & t_2 & \cdots & t_{d-r+1}\\
\vdots &\vdots & \ddots & \vdots \\
t_{r-2}& t_{r-1}& \cdots & t_{d-2}
\end{pmatrix} 
+ \begin{pmatrix} r & & & &&&&\\
& \!\!\! r{-}2 & & &&&&\\
&&\ddots&&&&&\\
& & \!\!& \!\!3 \!& &&&\\
& &&\!&\!\!\!-1\!&&& \\
&&&&& \!1 &&\\
&&&&&\!&\!\!\ddots & \\
&&&&&&\!\!&\!\! \!\! r{-}4\,
\end{pmatrix}\begin{pmatrix}t_2 & t_3 & \cdots & t_{d-r+2}\\
t_3 & t_4 & \cdots & t_{d-r+3}\\
\vdots & \vdots & \ddots & \vdots \\
t_{r}& t_{r+1}& \cdots & t_{d}
\end{pmatrix}.
\end{align*}
\end{conjecture}

We do not know yet what the general formula for $\mathcal M_r$ should be when $r$ is even.
The following  systematic construction 
led to the matrices $M_r$ and $\mathcal{M}_r$ in all cases known to us.
 Let $\widetilde P$
and $\widetilde Q$ be the multilinear forms in Lemma \ref{lem:multilinear}
that define $\mathcal{G}_{r,2}$. 
Let $F_j$ denote  the polynomial of degree $2r-2$  obtained by
eliminating $v_{1j}$ and $v_{2j}$ from $\widetilde P$ and $\widetilde Q$.
Let $G_j$ denote the product of all $\binom{r-1}{2}$
minors $v_{1k} v_{2l} - v_{1l} v_{2k}$ of $V$ where
$j \not\in \{k,l\}$. Each product
$F_j G_j$ is a polynomial of degree
$r(r-1)$.
Note that $F_r $ is $m_{11} m_{22} - m_{12} m_{21}$ in
(\ref{eq:eliminant}), and
$F_r G_r$ is  (\ref{eq:amazingfact}).
Now, the ideal $\langle F_1 G_1, F_2 G_2, \ldots, F_r G_r \rangle $
is Cohen-Macaulay of codimension $2$. By the Hilbert-Burch Theorem,
the $F_j G_j$ are the maximal minors of an  $(r-1) \times r$-matrix $M_r$, which can be extracted from the minimal free resolution of $\langle F_1 G_1, \ldots, F_r G_r \rangle $.
This is precisely our matrix. In order to extend 
Theorem \ref{thm:sec3main}, and to find the desired
$\mathcal{M}_r$  for even~$r$, we need 
that all entries of the Hilbert-Burch matrix $M_r$ have the same degree $r$.

\begin{remark} \label{rem:howitrelates}     \rm
The singular locus of $\mathcal{T}_{r,2,d}$ is defined
by the $(r{-}2) \times (r{-}2)$-minors of $\mathcal{M}_r$.
It would be interesting to study this subvariety of $\mathbb{P}^d$
and how it relates to singularities of $\mathcal{F}_{r,2}$.
For instance, for $r=4$, this singular locus is precisely the
odeco variety $\mathcal{T}_{2,2,d} $, and,
using \cite[Proposition 3.6]{Rob}, we can see that its prime ideal
is generated by the $2 \times 2$-minors of $\mathcal{M}_4$.
\end{remark}

\smallskip
In Section \ref{sec5} we shall see how the matrices $\mathcal{M}_r$
can be used to find a frame decomposition of a given symmetric
$2 {\times} 2 {\times} \cdots {\times} 2$-tensor $T$.
We close with an example that shows how this task differs 
from the easier problem of constructing a rank $r$ Waring decomposition of $T$.

\begin{example} \label{ex:notcontain} \rm
Let $r = 4$, $d = 8$, and consider the sum of two odeco tensors
$$ T \,\,=\,\, x^8 + y^8 \,+\, (x-y)^8 + (x+y)^8 \,\,=\,\, 3 x^8 + 56 x^6 y^2 + 140 x^4 y^4 + 56 x^2 y^6 + 3 y^8. $$
The coordinates of this tensor are
$t_0 = t_8 = 3$, $ t_2 = t_4 = t_6 = 2$, and $t_1 = t_3 = t_5 = t_7 = 0$.
Here, the $3 \times 5$-matrix $\mathcal{M}_4$ has rank $2$.
This verifies  that $T$ lies in $\mathcal{T}_{4,2,8}$, in accordance
with Example \ref{eq:vierzwei}.
However, the $4 \times 4$-matrix $\mathcal{M}_5$ 
is invertible.
This means that $T$ does not lie in $\mathcal{T}_{5,2,8}$.
In other words, there is no funtf among the rank $5$ Waring decompositions of $T$.
\hfill $\diamondsuit$
\end{example}

\section{Ternary Forms and Beyond}
\label{sec4}

We now move on to higher dimensions $n \geq 3$.
Our object of study is the fradeco variety
$$ \mathcal{T}_{r,n,d} \,\subset \, \PP({\rm Sym}_d(\CC^n)) . $$
A very first question is: What is the dimension of $\mathcal{T}_{r,n,d}\,$?
In Lemma \ref{lem:fradim}, we saw that
$\,{\rm dim}(\mathcal{T}_{r,2,d}) = 2r-3$.
The following proposition generalizes that formula
to arbitrary $n$:

\begin{proposition}
For all $r >n$ and $d \geq 3$, the dimension of $\mathcal{T}_{r,n,d}$ is bounded above by
\begin{equation}
\label{eq:expdim} {\rm min} \left\{(n-1)(r-n) + \frac{(n-1)(n-2)}{2} + r-1\, ,\,\,
\binom{n+d-1}{d}-1 \,\right\} . 
\end{equation}
\end{proposition}

\begin{proof}
The right number is the dimension of the ambient space, so this is an upper bound.
The left number is the dimension of  $\mathcal{F}_{r,n} \times \PP^{r-1}$, by
the formula in Theorem~\ref{thm:strawn}. The formula
(\ref{eq:Trank2}) expresses our variety as the (closure of the) image of a polynomial map
\begin{equation}
\label{eq:hopefullybirational}
 \mathcal{F}_{r,n} \times \PP^{r-1} \,\longrightarrow \, \mathcal{T}_{r,n,d}  .
 \end{equation}
The dimension of the image of this map is bounded above by the
dimension of the domain.
\end{proof}

We conjecture that the true dimension always agrees with the expected dimension:

\begin{conjecture} \label{conj:dim}
The dimension of the variety  $\mathcal{T}_{r,n,d}$ is equal to (\ref{eq:expdim}) 
for all $r >n$ and $d \geq 3$.
\end{conjecture}

This conjecture is subtler than it may seem.
Let $\sigma_r \nu_d \PP^{n-1}$ denote the Zariski closure of the set of tensors of rank $\leq r$ in 
$\PP({\rm Sym}_d(\CC^n))$. Geometrically,
this is the $r$-th secant variety of the $d$-th Veronese
embedding of $\PP^{n-1}$. It is known that $\sigma_r \nu_d \PP^{n-1}$
 has the expected dimension in almost all cases.
The Alexander-Hirschowitz Theorem (cf.~\cite{BrOt, Lan}) states that,
assuming $d \geq 3$, the dimension of $\sigma_r \nu_d \PP^{n-1}$
 is lower than expected in precisely four cases:
\begin{equation}
\label{eq:AH_list}
 (r,n,d) \,\, \in \,\, \bigl\{
(5,3,4), \,
(7,5,3),\,
(9,4,4),\,
(14,5,4) \bigr\}.
\end{equation}
One might think that in these cases also the fradeco subvariety
$\mathcal{T}_{r,n,d}$ has lower than expected dimension. However,
the results summarized in Theorem 
\ref{thm:sec4main} suggest that this is not the case.

\begin{table}[h]
\[
\begin{array}{|c|c|c|c|c|}
 \hline
\textrm{variety} & \textrm{dim} & \textrm{codim} &\textrm{degree}  & \textrm{known equations} \\ \hline
\mathcal{T}_{4,3,3} & 6 & 3 &  17 &  \textrm{$3$ cubics, $6$ quartics} \\ \hline
\mathcal{T}_{4,3,4} & 6 & 8 & 74 &  \textrm{$6$ quadrics, $37$ cubics} \\ \hline
\mathcal{T}_{4,3,5} &  6 & 14  & 191 &  \textrm{$27$ quadrics, $104$ cubics}  \\  \hline
\mathcal{T}_{5,3,4} & 9 & 5 &  210 &   \textrm{$1$ cubic, $6$ quartics} \\ \hline
\mathcal{T}_{5,3,5} & 9 & 11 & 1479 & \textrm{$20$ cubics, $213$ quartics}  \\ \hline
\mathcal{T}_{6,3,4} & 12 & 2 &  99  & \textrm{none in degree $\leq 5$} \\ \hline
\mathcal{T}_{6,3,5} & 12 & 8 &  4269 &\textrm{one quartic} \\ \hline
\mathcal{T}_{7,3,5} &  15 &  5 & \geq 38541 &  \textrm{none in degree $\leq 4$}   \\  \hline
\mathcal{T}_{8,3,5} &  18 &  2  & 690  &     \textrm{none in degree $\leq 5$}\\  \hline
\mathcal{T}_{10,3,6} & 24 & 3  & \geq 16252 &  \textrm{none in degree $\leq 7$}  \\  \hline
\mathcal{T}_{5,4,3} & 10 & 9 &  830 & \textrm{none in degree $\leq 4$} \\ \hline
\mathcal{T}_{6,4,3} & 14 & 5 & 1860 &  \textrm{none in degree $\leq 3$} \\ \hline
\mathcal{T}_{7,4,3} & 18 & 1 & 194 & \textrm{one in degree $194$} \\ \hline
\end{array}
\]
\caption{\label{tab:fradeco} A census of small fradeco varieties}
\end{table}

\begin{theorem} \label{thm:sec4main}
Consider the fradeco varieties $\mathcal{T}_{r,n,d}$
 in the cases when $n \geq 3$ and $\,1 \leq {\rm dim}(\mathcal{T}_{r,n,d} ) \cdot 
{\rm codim}(\mathcal{T}_{r,n,d}) \leq 100 $.
Table \ref{tab:fradeco} gives their degrees and some defining polynomials.
The last column shows the minimal generators of lowest possible degrees
in the ideal of $\mathcal{T}_{r,n,d}$.
\end{theorem}

\begin{proof}[Computational Proof]
The dimensions are consistent with Conjecture~\ref{conj:dim}.
They were verified by computing tangent spaces
at a generic point
using {\tt Bertini} and {\tt Matlab}. The degrees were computed 
with the monodromy loop method described in Subsection \ref{subsub1}.
The numerical Hilbert function  method in Subsection \ref{subsub2} was used to 
determine how many polynomials of a given degree vanish on $\mathcal{T}_{r,n,d}$.
This was followed up with computations in exact arithmetic in
{\tt Maple} and {\tt Macaulay2}. These confirmed the earlier numerical results,
and they enabled us to find the explicit polynomials
in $\QQ[T]$ that are listed in Examples
\ref{ex:T433}, \ref{ex:T434} and \ref{ex:T534}.  In the cases where we report no equations occurring below a certain degree, this is a combination of Corollary~\ref{prop:bound} and the numerical Hilbert function computation. 
\end{proof}

We shall now discuss some of the cases appearing
in Theorem \ref{thm:sec4main} in more detail.

\begin{example} \label{ex:T433} \rm
The $6$-dimensional variety $\mathcal{T}_{4,3,3} \subset
\mathbb{P}^9$ has the parametrization
\begin{equation}
\label{eq:433para}
 \begin{matrix}
t_{300} & = & v_{11}^3+v_{12}^3+v_{13}^3+v_{14}^3, \\ 
t_{030} & = & v_{21}^3+v_{22}^3+v_{23}^3+v_{24}^3 ,\\ 
t_{003} & = & v_{31}^3+v_{32}^3+v_{33}^3+v_{34}^3 ,\\ 
t_{012} & = & v_{21}v_{31}^2 +  v_{22}v_{32}^2 +  v_{23}v_{33}^2 +  v_{24} v_{34}^2 ,\\ 
t_{021} & = & v_{21}^2 v_{31} + v_{22}^2v_{32} +  v_{23}^2v_{33} + v_{24}^2 v_{34} ,\\ 
t_{102} & = & v_{11} v_{31}^2 + v_{12}v_{32}^2 + v_{13}v_{33}^2+ v_{14}v_{34}^2 ,\\ 
t_{120} & = & v_{11}v_{21}^2  + v_{12}v_{22}^2  + v_{13}v_{23}^2  +  v_{14}v_{24}^2 ,\\ 
t_{201} & = & v_{11}^2 v_{31} + v_{12}^2 v_{32} + v_{13}^2 v_{33} + v_{14}^2 v_{34}, \\ 
t_{210} & = & v_{11}^2 v_{21} + v_{12}^2 v_{22} + v_{13}^2 v_{23} + v_{14}^2 v_{24} ,\\ 
t_{111} & = & v_{11} v_{21} v_{31} + v_{12} v_{22} v_{32} 
+ v_{13} v_{23} v_{33} +  v_{14} v_{24} v_{34} .\\ 
\end{matrix}
\end{equation}
Here the matrix $V= (v_{ij})$ is given by the parametrization of
$\mathcal{G}_{4,3}$ seen in   (\ref{eq:ex434c}) of Example \ref{ex:opening}.

Using exact linear algebra in {\tt Maple}, we find that the ideal of $\mathcal{T}_{4,3,3}$
contains no quadrics, but it contains  three linearly independent  cubics and $36$ quartics.
One of the cubics is
\begin{equation}
\label{eq:april23}
  C_{123} + 2C_{145} + 2C_{345} - C_{126} - C_{236} - 4 C_{456},
\end{equation}
where $C_{ijk}$ denotes the determinant  of the $3\times 3$ submatrix with columns
$i,j,k$ in
$$
C = \begin{pmatrix}t_{300} & t_{210} &  t_{120} & t_{201} &  t_{111} & t_{102}\\
    t_{210} &  t_{120} &  t_{030} & t_{111} &  t_{021} &  t_{012}\\
    t_{201} &  t_{111} &  t_{021} & t_{102} & t_{012} & t_{003}\end{pmatrix}. $$
The other two cubics are obtained from this one by permuting the indices.
The resulting three cubics define a complete intersection in $\mathbb{P}^9$.
However, that complete intersection strictly contains $\mathcal{T}_{4,3,3}$
because the three cubics have only $30$ multiples in degree $4$, whereas 
we know that  $36$ quartics vanish on $\mathcal{T}_{4,3,3}$.
Using {\tt Macaulay2}, we  identified six minimal ideal generators in degree $4$,
and we found that the nine known generators
generate a Cohen-Macaulay ideal of codimension $3$ and 
degree $17$. Using {\tt Bertini}, we independently verified that
fradeco variety $\mathcal{T}_{4,3,3}$ has degree $17$. 
This implies that we have found the  correct prime ideal.
\hfill $\diamondsuit $
\end{example}

\begin{example}  \label{ex:T434} \rm
The variety $\mathcal{T}_{4,3,4}$ is also $6$-dimensional, and it lives
in the $\PP^{14}$ of ternary quartics. The parametrization is
as in (\ref{eq:433para}) but with quartic monomials instead of cubic.
Among the ideal generators for $\mathcal{T}_{4,3,4}$
are six quadrics and $37$ cubics. One of the quadrics is
$$
\begin{matrix}
8(t_{013}^2{-}t_{004} t_{022})
+8(t_{031}^2{-}t_{022} t_{040})
+8(t_{211}^2{-} t_{202} t_{220})
+18(t_{112}^2{-}t_{103} t_{121})
+18(t_{121}^2{-}t_{112} t_{130})   \\
+ (t_{004} t_{040}{+}19 t_{022}^2 {-}20 t_{013} t_{031})
+ (t_{004} t_{220}{+}t_{022} t_{202} {-}2 t_{013} t_{211})
+ (t_{040} t_{202}{+}t_{022} t_{220}  {-}2 t_{031} t_{211}).
\end{matrix}
$$
A {\tt Bertini} computation suggests that the known generators
suffice to cut out $\mathcal{T}_{4,3,4}$. We also note that
the $27$ quadrics for $\mathcal{T}_{4,3,5}$ come
from the $6$ quadrics for $\mathcal{T}_{4,3,4}$. For instance,
replacing each variable $t_{ijk}$ by $t_{i,j,k+1}$ yields the quadric
$8 t_{014}^2+8 t_{032}^2 + \cdots +19 t_{023}^2$
for $\mathcal{T}_{4,3,5}$.~$\diamondsuit $
\end{example}

\begin{example} \label{ex:T534} \rm
The fradeco variety $\mathcal{T}_{5,3,4}$ is especially
interesting because $(5,3,4)$ appears on the Alexander-Hirschowitz list
(\ref{eq:AH_list}). The unique cubic that vanishes on $\mathcal{T}_{5,3,4}$ is
$$ \begin{matrix}
46 t_{022} t_{202} t_{220}
+73 t_{112} t_{121} t_{211}
-4 t_{004} t_{040} t_{400}
+19 [t_{013} t_{130} t_{301}]_2
-50 [t_{004} t_{112}^2]_3
-22 [t_{004} t_{220}^2]_3 \\
-18 [t_{022} t_{211}^2]_3
+50 [t_{004} t_{022} t_{202}]_3
+26 [t_{004} t_{130} t_{310}]_3
+100 [t_{013} t_{103} t_{112}]_3
-53 [t_{013} t_{121} t_{310}]_3 \\
+5 [t_{004} t_{022} t_{400}]_6
-50 [t_{013}^2 t_{202}]_6
-5 [t_{013}^2 t_{220}]_6
+45 [t_{004} t_{031} t_{211}]_6
-40 [t_{022} t_{202}^2]_6
+5 [t_{004} t_{022} t_{220}]_6 \\
+40 [t_{022} t_{112}^2]_6
-5 [t_{004} t_{130}^2]_6
-45 [t_{004} t_{121}^2]_6 
-10 [t_{004} t_{112} t_{130}]_6
{-}45 [t_{013} t_{022} t_{211}]_6
{+}35 [t_{013} t_{031} t_{202}]_6 \\
+10 [t_{013} t_{103} t_{130}]_6
+10 [t_{013} t_{112} t_{121}]_6
-80 [t_{013} t_{112} t_{301}]_6
+80 [t_{013} t_{202} t_{211}]_6
+8 [t_{013} t_{211} t_{220}]_6.
\end{matrix}
$$
This polynomial has $128$ terms:
each bracket denotes an orbit of monomials under the $S_3$-action,
and the subscript is the orbit size. 
 In addition, six fairly large quartics
vanish on $\mathcal{T}_{5,3,4}$.
The seven known generators cut out a
reducible variety of dimension $9$ in $\PP^{14}$.
The fradeco variety $\mathcal{T}_{5,3,4}$ is the unique
top-dimensional component.  But,
using {\tt Bertini}, we found two extraneous
components of dimension $7$.
Their degrees are $120$ and $352$ respectively.
\hfill $\diamondsuit$
\end{example}

We close this section by examining the geometric
interplay between fradeco varieties and secant varieties.
We write $\sigma_r \nu_d \PP^{n-1}$ for the
$r$-th secant variety of the $d$-th Veronese embedding of $\PP^{n-1}$.
This lives in $\PP({\rm Sym}_d(\CC^n))$ and comprises
rank $r$ symmetric tensors (\ref{eq:Trank2}).
The same ambient space contains  the fradeco variety $\mathcal{T}_{r,n,d}$
 and all its secant varieties $\sigma_s \mathcal{T}_{r,n,d}$.

\begin{theorem} \label{thm:secant}
For any $r > n \geq d \geq 2$, we have
\begin{equation}
\label{eq:secant1}
  \sigma_{r-n} \nu_d \PP^{n-1}\, \subset \, \mathcal{T}_{r,n,d} \,\subset\, \sigma_r \nu_d \PP^{n-1}, 
  \end{equation}
and hence $\, \mathcal{T}_{r-n,n,d} \subset \, \mathcal{T}_{r,n,d}\,$
whenever $r \geq 2n$. Also, if $r = r_1 r_2$ with $r_1 \geq 2$ and $r_2 \geq n$, then
\begin{equation}
\label{eq:secant2}
 \sigma_{r_1} \mathcal{T}_{r_2,n,d}\,\, \subseteq \,\, \mathcal{T}_{r,n,d} . 
 \end{equation}
\end{theorem}

\begin{proof} We fix $d$.
The right inclusion in (\ref{eq:secant1}) is immediate from the definition.
For the left inclusion we use the 
parametrization of $\mathcal{F}_{r,n}$ given in
(\ref{eq:strawpara2}) and (\ref{eq:strawpara3}).
The point is that the $(r-n) \times n$-matrix $W$ can be chosen freely.
Equivalently, the projection of $\mathcal{G}_{r,n} \subset (\PP^{n-1})^r$ to any
coordinate subspace $(\PP^{n-1})^{r-n}$ is dominant.
This means that the first $r-n$ summands in (\ref{eq:Trank1})
are arbitrary powers of linear forms, and this establishes the left inclusion
in (\ref{eq:secant1}).

To show the inclusion (\ref{eq:secant2}), we consider
arbitrary frames $V_1,V_2,\ldots,V_{r_1} \in \mathcal{F}_{r_2,n}$.
Then the $n \times r$-matrix $V = (V_1,V_2,\ldots,V_{r_1})$ is
a frame in  $\mathcal{F}_{r,n}$.
Each $V_i$ together with a choice of $\lambda_i \in \RR^{r_2}$ 
determines a point on $\mathcal{T}_{r_2,n,d}$. 
Thus we have $r_1$ points in $\mathcal{T}_{r_2,n,d}$,
and any point on the $\PP^{r_1-1}$ spanned by these
lies in $\mathcal{T}_{r,n,d}$, where it is represented by 
$V$ with $\lambda = (\lambda_1, \ldots,\lambda_{r_1}) \in \RR^r$.
\end{proof}

\begin{example} \rm
Let $n = 2$ and write $H = (t_{i+j})$ for a Hankel matrix of unknowns with 
$r+1$ rows and sufficiently many columns.
The secant variety $\sigma_r \nu_d \PP^{1}$ is defined by the ideal
$I_{r+1}(H)$ of $(r{+}1) \times (r{+}1)$-minors of $H$. The
ideal-theoretic version of (\ref{eq:secant1}) states that
$$ I_{r-1}(H) \,\supset\, I_{r-1}(\mathcal{M}_r) \, \supset \, I_{r+1}(H) .$$
It is instructive to check this.
The left inclusion follows from the Cauchy-Binet Theorem
applied to $\mathcal{M}_r = A \cdot H $ where $A$ is 
the  $(r{-}1) \times (r{+}1)$ integer matrix underlying $M_r$.
\hfill $\diamondsuit$
\end{example}

\begin{remark} \rm
\begin{itemize} \item[(a)] Since concatenations of frames in $\RR^n$ are always frames,
 (\ref{eq:secant2}) generalizes from secant varieties to joins.
 Namely, if $r = r_1+ r_2$, then
  $\mathcal{T}_{r_1,n,d} \star \mathcal{T}_{r_2,n,d} \subset \mathcal{T}_{r,n,d}$. 
  \vspace{-0.1in}
\item[(b)] The inclusion in (\ref{eq:secant2}) is always strict, with one
notable exception: $\, \sigma_2 \mathcal{T}_{2,2,d} \, = \, \mathcal{T}_{4,2,d}$.
\end{itemize}
\end{remark}

Theorem~\ref{thm:secant} implies that
the Veronese variety $\nu_d \PP^{n-1}$ is contained
in the fradeco variety $\mathcal{T}_{r,n,d}$  with $r > n$.
This is illustrated in Example~\ref{ex:T434} where we
wrote the quadric that vanishes on $\mathcal{T}_{4,3,4}$
as a linear combination of the 
 binomials that define $\nu_4 \PP^2 \subset \PP^{14}$.
 The formula (\ref{eq:april23}) shows that this cubic vanishes on $\,\sigma_2 \nu_3 \PP^2$.
 Similarly, we can verify that the cubic in
 Example~\ref{ex:T534} vanishes on $\sigma_2 \nu_4 \PP^2$ by writing
 it as a linear combination of the $3 \times 3$-minors $C_{ijk,lmm}$ of
the $6 \times 6$-catalecticant $C$ matrix in (\ref{eq:sixbysix}). One such 
expression  is
\small \[\begin{matrix}
50C_{012,012}-30C_{012,123}+50C_{012,034}-30C_{012,125}+50C_{012,045}+63C_{012,345}-10C_{013,024}+10C_{013,234}\\
+5C_{013,015}+35C_{013,135}+34C_{013,245}+5C_{023,023}-80C_{023,134}+5C_{023,025}-26C_{023,235}-19C_{023,145}\\
-30C_{123,123}+29C_{123,125}-10C_{123,345}-10C_{014,025}+19C_{014,235}-53C_{014,145}-30C_{024,245} +5C_{034,034}\\
+26C_{034,045}+5C_{034,345}+50C_{134,134}+50C_{134,235}+30C_{134,145}+30C_{234,245}+5C_{015,015}+26C_{015,135}\\
+50C_{015,245}-5C_{025,235}-10C_{025,145}-10C_{125,345}-4C_{035,035}+5C_{135,135}+50C_{135,245}+5C_{235,235}\\
+5C_{045,045}+5C_{045,345}+50C_{245,245}
. \end{matrix}
\]

Theorem~\ref{thm:secant} gives lower bounds on the degrees of the equations defining fradeco varieties:

\begin{corollary}\label{prop:bound}
All non-zero polynomials in the ideal of $\mathcal{T}_{r,n,d}$  must have degree at least $r-n+1$.
\end{corollary}

\begin{proof}
The ideal of the Veronese variety $\nu_{d}\PP^{n-1}$ contains no linear forms. It is
 generated by $2\times 2$ minors of catalecticants. A general result on
 secant varieties \cite[Thm.~1.2]{SidmanSullivant} implies
 that the ideal of $\sigma_{r-n}\nu_{d} \PP^{n-1}$ is zero in degree $\leq r-n$. The inclusion 
$  \sigma_{r-n} \nu_d \PP^{n-1}\, \subset \, \mathcal{T}_{r,n,d}$ yields the claim.
\end{proof}

In Table \ref{tab:fradeco} we see that $\mathcal{T}_{4,3,4}$, 
$\mathcal{T}_{4,3,5}$, $\mathcal{T}_{5,3,4}$, $\mathcal{T}_{5,3,5}$ 
and $\mathcal{T}_{6,3,5}$ have their first minimal generators in the lowest possible degrees.
However this is not always the case, as shown dramatically by $\mathcal{T}_{7,4,3}$.

\section{Numerical Recipes}
\label{sec5}

Methods from Numerical Algebraic Geometry (NAG) are useful for studying
the decomposition of tensors into frames. Many of the results on
fradeco varieties $\mathcal{T}_{r,n,d}$ reported in Sections \ref{sec3} and \ref{sec4}
were discovered using NAG. In this section we discuss the relevant methodologies.
Our experiments
 involve a mixture of using \texttt{Bertini} \cite{BertiniSoftware}, \texttt{Macaulay2} \cite{M2}, \texttt{Maple}, and \texttt{Matlab}.  

All algebraic varieties have an {\em implicit representation}, as the solution set to a system of
polynomial equations. Some special varieties admit a {\em parametric representation},  as
the (closure of the) image of a map whose coordinates are rational functions. Having to pass
 back and forth between these two representations is a
ubiquitous task in  computational algebra.

The fradeco variety studied in this paper is given by a mixture of implicit and parametric.
Our point of departure is the implicit representation (\ref{eq:funtf}) of the
funtf variety  $\mathcal F_{r,n}$, or its homogenization
$\mathcal{G}_{r,n}$. Built on  top of that is the
parametrization (\ref{eq:Trank1}) of rank $r$ tensors:
\begin{equation}
\label{eq:representation}
\begin{matrix}
\CC^{n \times r} \times \CC^{r} &&  \mathrm{Sym}_{d}(\CC^{n}) \\
\cup && \cup \\
 \mathcal{F}_{r,n} \times  \CC^{r} &
\overset{\Sigma_{d}}{\longrightarrow}
 & \widehat{\mathcal{T}}_{r,n,d} \\\\
 (V, \lambda ) & \longmapsto & \lambda_{1} {\bf v}_{1}^{ \otimes d} 
 + \lambda_{2} {\bf v}_{2}^{\otimes d} 
 + \cdots + \lambda_{r} {\bf v}_{r}^{\otimes d}
\end{matrix}
\end{equation}
Here, $ \widehat{\mathcal{T}}_{r,n,d}$ denotes the affine cone
over the projective variety $\mathcal{T}_{r,n,d}$.
The input to our {\em decomposition problem} is an arbitrary
symmetric $n {\times} n {\times} \cdots {\times} n$-tensor $T$ and a positive integer~$r$.
The task is to decide whether $T$ lies in $\widehat{\mathcal{T}}_{r,n,d}$, and, if yes, to compute
a preimage $(V, \lambda)$ under the map $\Sigma_d$ in
(\ref{eq:representation}). Any preimage must
satisfy the non-trivial constraint $V \in \mathcal{F}_{r,n}$.

\subsection{Decomposing fradeco tensors}

We discuss three approaches to finding frame decompositions of symmetric tensors.

\subsubsection{Tensor power method}
Our original motivation for this project came from the case $r=n$
of odeco tensors \cite{Rob}. If $T \in \widehat{\mathcal{T}}_{n,n,d}$,
then the {\em tensor power method} of \cite{AGHKT}  reliably reconstructs
the decomposition (\ref{eq:Trank1}) where
$\{{\bf v}_1,\ldots,{\bf v}_n\}$ is an orthonormal basis of $\RR^n$.
The algorithm is to iterate the rational map
 $\nabla T: \PP^{n-1} \dashrightarrow \PP^{n-1}$
given by the gradient vector $\nabla T = (\partial T / \partial x_1, \ldots, \partial T /\partial x_n)$.
This map is regular when the hypersurface $\{T = 0\}$ is smooth.
The fixed points of $\nabla T$ are the {\em eigenvectors} of the tensor $T$.
Their number was given in \cite{CStu}. The punchline is this: if the multipliers  $\lambda_1,\ldots,\lambda_n$ 
in (\ref{eq:Trank1}) are positive, then ${\bf v}_1,\ldots,{\bf v}_n$
are precisely the {\em robust eigenvectors}, i.e.~the attracting fixed points of 
the gradient map $\nabla T$.

This raises the question whether the tensor power method also works
for fradeco tensors. The answer is ``no'' in general, but it is
 ``yes'' in some special cases.
 
 \begin{example} \label{ex:eigenvec} \rm
 Let $n=2,r=4,d=5$ and consider the fradeco quintic
 $$ T \,=\, \alpha x^5 + y^5 + (x+y)^5 + (x-y)^5 \quad \in \,\,\mathcal{T}_{4,2,5}, $$
 where $\alpha > 6 $ is a parameter.
 The eigenvectors of $T$ are the zeros in $\PP^1$ of the binary quintic
 $$ y \frac{\partial T}{\partial x} - x \frac{\partial T}{\partial y} \,\,\,= \,\,\,
   5 y \cdot \biggl(
    (\alpha x-6)x^4 \,+\, \bigl(2xy-\frac{1}{4}y^2 \bigr)^2 \,+\, \frac{31}{16}y^4  \biggr).
  $$
 The point $(1:0)$ is an eigenvector, but
 there are no other real eigenvectors, as the expression is a sum of squares.
    Hence
 the frame decomposition of $T$ cannot be recovered from its eigenvectors.
 \hfill $\diamondsuit$
 \end{example}
 
\begin{example} \label{ex:tensor power} \rm
For any reals $\lambda_1,\lambda_2,\lambda_3,\lambda_4 > 0$
and any integer $d \geq 5$, we consider the tensor
 \begin{equation}
 \label{eq:ex434d}
T \,= \,
  \lambda_1 (-5x_1+x_2+x_3)^d
+ \lambda_2 (x_1-5x_2+x_3)^d
+ \lambda_3 (x_1+x_2-5x_3)^d
+ \lambda_4 (3x_1 + 3x_2 + 3x_3)^d.
\end{equation}
This tensor has precisely four robust eigenvectors, namely
the columns of the matrix $V$ in (\ref{eq:appscaled}).
Hence the frame decomposition of $T$ can be recovered 
by the tensor power method.
\hfill $\diamondsuit$
\end{example}

The following conjecture generalizes this example.

\begin{conjecture} \label{thm:tensorpowerworks}
Let $ r = n+1 < d $ and $T \in \mathcal{T}_{n+1,n,d}$ with
$\lambda_1,\ldots,\lambda_{n+1} > 0$ in (\ref{eq:Trank1}).
Then ${\bf v}_1,\ldots,{\bf v}_{n+1}$
are the robust eigenvectors of $T$, so they are found by the
tensor power method.
\end{conjecture}

Example \ref{ex:eigenvec} shows that 
Conjecture \ref{thm:tensorpowerworks} is false for $r \geq n+2$,
and it suggests that the Tensor Power Method will not work in general.
 We next discuss two alternative approaches.

\subsubsection{Catalecticant method for frames}

The matrices in Theorem \ref{thm:sec3main} furnish
a practical algorithm for the decomposition problem when $n=2$.
This is a variant of Sylvester's  Catalecticant Algorithm,
and it works as follows.

Our input is a binary form $T\in\widehat{\mathcal{T}}_{r,2,n}$.
 We seek to recover the tight frame into which $T$ decomposes. Since we do
 not know  the fradeco rank $r$ in advance, we start with $\mathcal M_3, \mathcal M_4, \mathcal M_5,$ etc. and plug in the coordinates $t_i$ of $T$.  The fradeco rank is 
 the first index $r$ with $\mathcal{M}_r$ rank deficient.
 
 
If the matrix $\mathcal M_r$ is rank deficient, then its rank is at most $r-2$. Let us assume that the rank equals exactly $r-2$. Otherwise $T$ is a singular point (cf.~Remark~\ref{rem:howitrelates}).
Then, up to scaling, we find the unique row vector ${\bf w} \in \RR^{r-1}$ in the left kernel of $\mathcal{M}_r$. By Theorem \ref{thm:sec3main} we know that $\mathcal M_r$ is the product of the matrix $M_r$ and an $(r-1)\times(d-r-1)$ matrix with entries $v_{i1}^{d-r-j+1}v_{i2}^{j-1}$, where $V = (v_{ij}) \in \mathcal{G}_{r,2}$ is the desired frame. Moreover, the matrix $M_r$ has rank $r-2$, so the vector $\bf w$ also lies in the left kernel of $M_r$, i.e. ${\bf w} \cdot M_r = 0$.
Thus,
$$ 0 = {\bf w} \cdot M_r \,\, = \,\, 
\bigl(
 f(v_{11},v_{21}),
 f(v_{12},v_{22}),
\ldots,
 f(v_{1r},v_{2r}) \bigr) ,$$
 where $f(x,y)$ is a binary form of degree $r$.
  The $r$ roots of $f(x,y)$ in $\PP^1$ are the columns of the desired
 $V = (v_{ij}) \in \mathcal{G}_{r,2}$.
Using these $v_{ij}$, the given binary form has the decomposition 
$$ T(x,y) \,\, = \,\,  \sum_{j=1}^r \lambda_j (v_{1j} x  + v_{2j} y)^d , $$
where the multipliers $\lambda_1,\ldots,\lambda_r$ are recovered
by solving a linear system of equations.

\begin{example} \rm
Let $r= 5$ and $d=8$. We illustrate this method for the binary octic
$$ \begin{small} \begin{matrix}
T = (-237{-}896 \alpha) x^8
+8 (65{+}241 \alpha)x^7y-28(16{+}68 \alpha)x^6y^2+56(5{+}31 \alpha )x^5y^3 
+70(2{-}56 \alpha)x^4y^4 \\
+56(-7+193 \alpha)x^3y^5+28(32-716 \alpha)x^2y^6
+8(-115+2671\alpha)x y^7+(435-9968 \alpha) y^8,
\end{matrix}
\end{small}
 $$
 where $\alpha = \sqrt{3} - 2$.
The parenthesized expressions are the coordinates $t_0,\ldots,t_8$. We find
$$\mathcal M_5 \,\,= \,\, 
\begin{pmatrix}-13548\alpha+595 & 3636\alpha-150 & -996\alpha+42 & 348\alpha+18\\
 2092\alpha-94& -548\alpha+26& 100\alpha-22& 148\alpha+50\\
 -2092\alpha+94& 548\alpha-26& -100\alpha+22& -148\alpha-50\\
 996\alpha-30& -348\alpha-6& 396\alpha+90& -1236\alpha-317
 \end{pmatrix}.
$$
This matrix has rank $3$ and its left kernel is the span of the vector
$\mathbf{w} = (0,1,1,0)$.
Therefore, 
\scriptsize
$$0 \,=\, \mathbf w M_5 \,= \,\begin{pmatrix} 0\\ 1\\ 1\\ 0 \end{pmatrix}^{\!\! T} 
\begin{pmatrix} v_{21}^5 + 5v_{11}^5 & v_{22}^5 
+ 5v_{12}^5 & v_{23}^5 + 5v_{13}^5 & v_{24}^5 + 5v_{14}^5 & v_{25}^5 + 5v_{15}^5\\
v_{11}v_{21}^4 - 3v_{11}^3v_{21}^2 & v_{12}v_{22}^4 - 3v_{12}^3v_{22}^2 & v_{13}v_{23}^4 - 
3v_{13}^3v_{23}^2 & v_{14}v_{24}^4 - 3v_{14}^3v_{24}^2 & v_{15}v_{25}^4 - 3v_{15}^3v_{25}^2\\
3v_{11}^2v_{21}^3 - v_{11}^4v_{21} & 3v_{12}^2v_{22}^3 - v_{12}^4v_{22} & 3v_{13}^2v_{23}^3 - v_{13}^4v_{23} & 3v_{14}^2v_{24}^3 - v_{14}^4v_{24} & 3v_{15}^2v_{25}^3 - v_{15}^4v_{25} \\
5v_{11}^3v_{21}^2 + v_{11}^5 & 5v_{12}^3v_{22}^2 + v_{12}^5 & 5v_{13}^3v_{23}^2 
+ v_{13}^5 & 5v_{14}^3 v_{24}^2 + v_{14}^5 & 5v_{15}^3v_{25}^2 + v_{15}^5
\end{pmatrix}.
$$
\normalsize
Hence the five columns of the desired  tight frame $V = (v_{ij})$ are
the distinct zeros in $\PP^1$ of
$$f(v_{1i}, v_{2i}) \,\,= \,\, v_{1i}v_{2i}^4-3v_{1i}^3v_{2i}^2+3v_{1i}^2v_{2i}^3-v_{1i}^4v_{2i} \qquad
\hbox{for $i=1,\ldots,5$.} $$
We find
$$V \,=\, \begin{pmatrix} 1& 0 & 1 & \alpha & 1\\ 
0 & 1 & 1 & 1 & \alpha
\end{pmatrix} \,\, \in\,\mathcal G_{5, 2}.$$
It remains to solve the linear system of nine equations in
$\lambda = (\lambda_1,\ldots,\lambda_5)$ given by
$$T \,\,= \,\, \lambda_1 x^8 + \lambda_2 y^8 + \lambda_3 (x+y)^8  +
\lambda_4 (\alpha x + y)^8 + \lambda_5 (x +\alpha y)^8. $$
The unique solution to this system is
$\lambda_1 = \lambda_2 = \lambda_3 = \lambda_5=1$ and
$\lambda_4 = 1552 + 896 \sqrt 3$.
 \hfill $\diamondsuit$
\end{example}

\subsubsection{Waring-enhanced frame decomposition}

We now examine the decomposition problem for $n \geq 3$.
Since no determinantal representation of $\mathcal{T}_{r,n,d}$ is known,
 a system of equations must be solved
to recover $(V,\lambda)$ from a given tensor in $\widehat{\mathcal{T}}_{r,n,d}$.
In some  special situations, we can approach this by taking advantage of known results
on Waring decompositions. For instance, in Example~\ref{ex:opening} 
the Waring decomposition  is already the frame decomposition.
Example \ref{ex:notcontain}  shows that this is an exceptional situation.

We demonstrate the ``Waring-enhanced'' frame decomposition for the ternary quartic
$$  \begin{matrix}
\sum_{i+j+k=4} \frac{24} {i! j! k!} t_{ijk} x^i y^j z^k \, = \,
 467 x^4{+}152 x^3y{+}1448 x^3 z{+}660 x^2 y^2{-}1488x^2yz{+}4020x^2z^2 
 {+}536xy^3 \\ \qquad \qquad \qquad
 {-}1992 x y^2 z{+}2352 x y z^2{+}944 x z^3{+}227 y^4{-}1000 y^3 z
{+}2148 y^2 z^2{-}1960 y z^3 {+}1267 z^4.
\end{matrix}
$$
Ternary quartics of rank $\leq 5$ form a hypersurface of degree $6$ in  $\PP^{14}$. 
The equation of this hypersurface is
the determinant of the $6 \times 6$ catalecticant matrix $C$.
Here the dimension is one less than expected;
this is the first entry in 
the Alexander-Hirschowitz list (\ref{eq:AH_list}).
For the given quartic,
\begin{equation}
\label{eq:sixbysix}
 C \, = \,
 \begin{small}
 \begin{bmatrix}
 t_{400} &  t_{310} &  t_{301} &  t_{220} &  t_{211} &  t_{202} \\
 t_{310} &  t_{220} &  t_{211} &  t_{130} &  t_{121} &  t_{112} \\
 t_{301} &  t_{211} &  t_{202} &  t_{121} &  t_{112} &  t_{103} \\
 t_{220} &  t_{130} &  t_{121} &  t_{040} &  t_{031} &  t_{022} \\
 t_{211} &  t_{121} &  t_{112} &  t_{031} &  t_{022} &  t_{013} \\
 t_{202} &  t_{112} &  t_{103} &  t_{022} &  t_{013} &  t_{004} 
 \end{bmatrix} \, = \,
 \begin{bmatrix}
 467 & 38 & 362 & 110 & \! -124 & 670 \\
  38 & 110 & \! -124 & 134 & \! -166 & 196 \\
362 & \! -124 & 670 & \! -166 & 196 & 236 \\
 110 & 134 & \! -166 & 227 &\! -250 & 358 \\
\! -124 &\! -166 & 196 & \! -250 & 358 & \! -490 \\
 670 & 196 & 236 & 358 & \! -490 & 1267
\end{bmatrix}\!.
\end{small}
 \end{equation}
 This matrix has rank $5$ and its kernel
 is spanned by the vector corresponding to the quadric
  $\, q = 14 u^2 -  uv  -  2uw -  4v^2 -  11vw - 10w^2 $.
  The points $(u:v:w)$ in $\PP^2$ that lie on 
  the conic $\{q = 0\}$ represent
  all   the linear forms $ux + vy+wz$
  that may appear in a   rank $5$ decomposition.
  
  Our task is to find five points on the conic $\{q = 0\}$ that
  form a frame $V \in \mathcal{G}_{5,3}$.
  This translates into solving a rather challenging system of polynomial
  equations. One of the solutions is
  $$
V \,\,= \,\, ({\bf v}_1,{\bf v}_2,{\bf v}_3,{\bf v}_4,{\bf v}_5) \,\, = \,\,
 \begin{pmatrix}
 -1 &   \,\,2\,  &  \phantom{-}2 &  \phantom{-}1 + 2 \sqrt{3} &   -1 + 2 \sqrt{3}\, \\
 \phantom{-} 2  &  \,\, 2  \, &  -1 & -2 + \sqrt{3}    &   2 + \sqrt{3}   \\
 \phantom{-} 0 & \,\, 1  \, &  -2 &      5               &        -5 
   \end{pmatrix}.
$$
The given ternary quartic has the frame decomposition  $\, 
{\bf v}_1^{\otimes 4} + 
{\bf v}_2^{\otimes 4} + 
{\bf v}_3^{\otimes 4} + 
{\bf v}_4^{\otimes 4} + 
{\bf v}_5^{\otimes 4} $.

\subsection{Exploring the fradeco variety}

The following tasks make sense for any variety $X \subset \PP^{N}$
arising in an applied context: 
(i) sample points on $X$, (ii) compute the dimension and degree of $X$, 
(iii) compute an irreducible decomposition of $X$, 
(iv) find a parametrization of $X$, (v) find some polynomials that vanish on $X$, 
(vi) determine polynomials that cut out $X$, (vii) find generators for the ideal of $X$.
Numerical algebraic geometry (NAG) furnishes tools for addressing these points.
In our study, $X$ is the fradeco variety $\mathcal{T}_{r,n,d}$.
We used NAG to find answers in some cases.
In what follows, we explain our computations. Particular
emphasis is placed on the results  reported in Section \ref{sec4}
for the degree and Hilbert function of $\mathcal{T}_{r,n,d}$.
All computations are carried out by working on the affine cone
$\widehat{\mathcal{T}}_{r,n,d} \subset {\rm Sym}_d(\CC^n)$.

\subsubsection{Dimension and degree}
\label{subsub1}

The dimension and degree of the affine variety $\widehat{\mathcal{T}}_{r,n,d}$
can be computed directly from the  mixed parametric-implicit 
representation in (\ref{eq:representation}).
The dimension can be found by selecting
a random point on $\mathcal{F}_{r,n} \times \RR^r$,
determining its tangent space via \cite{Str}, and then
taking the image of this tangent space via
the derivative of the map $\Sigma_d$.
The image is a linear subspace in ${\rm Sym}_d(\RR^n)$, and 
its dimension is found via the rank of its defining  matrix.
These matrices are usually given numerically, 
in terms of  points sampled from $\mathcal{F}_{r,n} $,
so we need to use singular value decompositions.

The computation of the degree is carried out using {\em monodromy}.
We obtained the results of Theorem~\ref{thm:sec4main} by
applying essentially the same technique as in \cite{HIS, HOOS}, adapted to our situation where the mapping is from an implicitly defined source. Here are some highlights of this method for 
$\widehat{\mathcal{T}}_{r,n,d}$. We performed these computations using \texttt{Bertini} and \texttt{MatLab}.

Let $c$ denote the codimension of $\widehat{\mathcal{T}}_{r,n,d}$,
as given by the formula in Conjecture~\ref{conj:dim}.
The degree of $\widehat{\mathcal{T}}_{r,n,d}$ is 
 the number  of points in the intersection with a
 random $c$-dimensional affine subspace of ${\rm Sym}_d(\CC^n)$.
 Here we represent the fradeco variety purely numerically, namely
 as the set of images of points $(V,\lambda)$ under
 the parametrization $\Sigma_{d}$ shown in \eqref{eq:representation}.
   This method verifies the dimension of $\widehat{\mathcal{T}}_{r,n,d}$ because 
   the intersection would be empty if the dimension were lower than expected.

As a first step, we compute a numerical irreducible decomposition of
the funtf variety~$\mathcal F_{r,n}$.
This also gives its degree and dimension, as shown in Table \ref{tab:one}.
In particular, we obtain degree-many points of $\mathcal{F}_{r,n}$ that lie in a random linear space of dimension equal to $\textrm{codim}( \mathcal{F}_{r,n})$.

We take $V$ to be one of these generic points in $ \mathcal{F}_{r,n}$,
we select a  random vector $\lambda \in \CC^r$, and we compute
the fradeco tensor $\Sigma_d(V,\lambda)$.
We also fix a random $c$-dimensional linear subspace
$\mathcal{R}$ of $\mathrm{Sym}_{d}(\CC^{n})$ and a
random point $U$ in the $c$-dimensional
affine space $\mathcal{R}+U$.

By construction, the affine cone
$\widehat{\mathcal{T}}_{r,n,d}$  and the
affine space $\mathcal{R}+U$ intersect in
 ${\rm deg}(\widehat{\mathcal{T}}_{r,n,d})$ many points in 
$\mathrm{Sym}_{d}(\CC^{n})$.
One of these points is  $\Sigma_d(V,\lambda)$.
Our goal is to discover all the other intersection points by sequences of
parameter homotopies that form monodromy loops.
Geometrically, the base space for these monodromies
is the vector space quotient  $\mathrm{Sym}_{d}(\CC^{n})/R$.

We fix two further random points $P_1$ and $P_2$ in 
$\mathrm{Sym}_{d}(\CC^{n})$. These represent residue classes modulo
the linear subspace $\mathcal{R}$.  The data we fixed now define a
 (triangular) monodromy loop
 \[\xymatrix{ &(\mathcal{R} + U) \cap \widehat{\mathcal{T}}_{r,n,d}  \ar[dr]
  \\ (\mathcal{R} + P_2) \cap \widehat{\mathcal{T}}_{r,n,d} \ar[ur]&&
  (\mathcal{R} +  P_1) \cap \widehat{\mathcal{T}}_{r,n,d} \ar[ll] }\]
 We use \texttt{Bertini} to perform each linear parameter homotopy.
This constructs a path  $(V_t,\lambda_t)$ in the parameter space.
Here $t$ runs from $0$ to $3$. We start at $(V_0,\lambda_0) = (V,\lambda)$,
the point $\Sigma_d(V_i,\lambda_i)$ lies in 
$(\mathcal{R} +  P_i) \cap \widehat{\mathcal{T}}_{r,n,d} $ for $i=1,2$,
and $\Sigma_d(V_3,\lambda_3)$ is back in
$(\mathcal{R} + U) \cap \widehat{\mathcal{T}}_{r,n,d}$.
With high probability,
$\Sigma_d(V_3,\lambda_3) \not= \Sigma_d(V,\lambda)$ holds,
and we have discovered a new point.
Then we iterate the process. 
Let   $\,S_{k}:= \{\Sigma_{d}(V,\lambda),\dots, \Sigma_{d}(V',\lambda')\}\,$
denote the subset of
$(\mathcal{R} + U) \cap \widehat{\mathcal{T}}_{r,n,d}$
that has been found after $k$ steps.
In the next monodromy loop we trace the paths of
 $S_{k}$ to produce $\tilde S_{k+1}$, the endpoints of  monodromy loops starting from $S_{k}$.
 Using \texttt{MatLab}, we then merge the point sets to form $S_{k+1} = S_{k} \cup \tilde S_{k+1}$.
 We repeat this process until no new points are found after 20 consecutive monodromy loops.  The 
 number of points in $S_{k}$ is   very strong numerical evidence for the degree of~$\mathcal{T}_{r,n,d}$.
At this point, one can also use the trace test \cite{SVW}
 with pseudowitness sets \cite{HS} to confirm that degree.
 
\subsubsection{Numerical Hilbert Function}
 \label{subsub2}

We wish to learn the polynomial equations that vanish on $\mathcal{T}_{r,n,d}$.
The set $I$ of all such polynomials is a homogeneous prime ideal
in the polynomial ring over $\QQ$ whose variables are the entries 
$t_{i_1 i_2 \cdots i_d}$ of an indeterminate tensor $T$.
We write this polynomial ring as 
$$ \QQ[T]  \,=\,  \bigoplus_{e \geq 0} \QQ[T]_e \quad \simeq \quad  {\rm Sym}_e ({\rm Sym}_d(\QQ^n)) \,=\,
\bigoplus_{e \geq 0}   {\rm Sym}_* ({\rm Sym}_d(\QQ^n)). $$
The space of all polynomials of degree $e $ in the ideal $I$ is the subspace
$$ I_e \,= \,I \cap \QQ[T]_e \quad \subset \quad \QQ[T]_e  \, \simeq \,  {\rm Sym}_e ({\rm Sym}_d(\QQ^n)) .$$
A natural approach is to fix some small degree $e$ and to ask for a 
$\QQ$-linear basis of $I_e$.

The dimensions of these vector spaces are organized into the {\em Hilbert function}
$$ 
\mathbb{N} \,\rightarrow \,\mathbb{N}, \,\, e \,\mapsto \,  {\rm dim}_\QQ(I_e). 
$$
We used {\tt Bertini} and {\tt Matlab} to 
determine specific values  of the Hilbert function. In some cases, an independent
{\tt Maple} computation was used to construct a basis for the $\mathbb{Q}$-vector space $I_e$.

Fix  values for $r, n, d $. As discussed above,
we can use the parametrization (\ref{eq:representation})
to produce many sample points $T = \Sigma_d(V,\lambda)$ on $\mathcal{T}_{r,n,d}$.
The condition $f(T) = 0$ translates into a linear equation in 
the coefficients of a given polynomial $f \in \QQ[T]_e$, and
$I_e$ is the solution space to these equations as $T$ ranges over $\mathcal{T}_{r,n,d}$.
We write these linear equations as a matrix whose
number of columns is
${\rm dim}(\QQ[T]_e) = \binom{ \binom{n+d-1}{d} + e-1}{e}$, 
and with one row per sample point $T$.
In practice we take enough sample points so that
$I_e$ is sure to equal the kernel of that matrix.

This procedure may be carried out in exact arithmetic over $\QQ$
when sufficiently many exact points can be found on $\mathcal{F}_{r,n}$.  When floating point approximations are used, some care is required in choosing
the appropriate number of points and a sufficient degree of precision. 
This numerical test can become inconclusive in high dimension due to these issues.
 Using floating point arithmetic and 30,000 points of $\mathcal{F}_{r,n}$ we obtained the 
   values listed in Table~\ref{tab:three}.
The blanks  indicate that we did not find conclusive evidence for the exact value of 
$\dim (I_{e})$ in that case. For 
$\mathcal{T}_{5,4,3}$, $\mathcal{T}_{6,3,4}$, $\mathcal{T}_{6,4,3}$, $\mathcal{T}_{7,3,5}$, and $\mathcal{T}_{8,3,5}$ we also found no conclusive numerical 
evidence for equations in degrees less than~$5$.

\begin{table}
\[\begin{array}{|c|c|c|c|c|c|} \hline
\diaghead{\theadfont ideal\;\; $\deg e$\;\;}{ideal}{$\deg e$}
&       2 & 3 & 4 & 5   & 6  \\ \hline
\dim \mathcal{I}(\mathcal{T}_{5,2,9})_{e} &  0 & 0 & 5 & 46 & 235 \\ \hline
\dim \mathcal{I}(\mathcal{T}_{4,3,4})_{e} &  6 &127 & 1093 & 5986 &\\ \hline
\dim \mathcal{I}(\mathcal{T}_{4,3,5})_{e} &  27 & 651 & 6370 & &\\ \hline
\dim \mathcal{I}(\mathcal{T}_{5,3,4})_{e} &  0 &1 & 21 & &\\ \hline
\dim \mathcal{I}(\mathcal{T}_{5,3,5})_{e} &  0 &20 & 633 & &\\ \hline
\dim \mathcal{I}(\mathcal{T}_{6,3,5})_{e} &  0 & 0 & 1  & &\\ \hline
\end{array}
\]
\caption{\label{tab:three}
Numerical computation of the Hilbert functions of fradeco varieties}
\end{table}

The calculation of $\dim (I_{e})$ is a numerical rank computation via singular value decomposition, so at least in principle it is possible to also extract a basis of $I_{e}$. However, in practice, round-off errors yield imprecise values for the coefficients of the basis elements of $I_{e}$.
This makes it difficult to reliably determine an exact $\QQ$-basis of $I_{e}$ by numerical methods.

To discover the explicit ideal generators
displayed in Sections \ref{sec3} and \ref{sec4}, we instead used
exact arithmetic in {\tt Maple}.  A key step was to produce
points in the funtf variety $\mathcal{F}_{r,n}$ that are defined over
low-degree extension of $\mathbb{Q}$,
and to map them carefully via $\Sigma_d$.
To accomplish this, we used the representation of 
$\mathcal{G}_{r,n}$ discussed in Section \ref{sec2}.
In our experiments, we found that the {\tt solve} command
in {\tt Maple} was able to handle dense linear systems with
up to $3,500$ unknowns.

\bigskip
\bigskip

\noindent
{\bf Acknowledgments.}
The authors were hosted by the Simons Institute for the Theory of Computing in Berkeley during
the Fall 2014 program {\em Algorithms and Complexity in Algebraic Geometry}.
We are grateful to Jon Hauenstein for helping us with {\tt Bertini}, and we thank
Aldo Conca, Fatemeh Mohammadi, Giorgio Ottaviani
and Cynthia Vinzant for valuable comments.
Elina~Robeva and Bernd~Sturmfels were supported by
US National Science Foundation grant DMS-1419018.

 \medskip
 
\begin{small}

\bigskip \bigskip

\footnotesize 
\noindent {\bf Authors' addresses:}

\smallskip

\noindent Luke Oeding,
Dept.~of Mathematics, Auburn University,
Auburn, Alabama 36849, USA, {\tt oeding@auburn.edu}

\smallskip

\noindent Elina Robeva and Bernd Sturmfels, Dept.~of Mathematics, University of California, Berkeley, CA 94720-3840, USA,
{\tt erobeva@berkeley.edu, bernd@berkeley.edu}

\end{small}

\begin{thebibliography}{10}
\newcommand{\arxiv}[1]{{\tt \href{http://arxiv.org/abs/#1}{{arXiv:#1}}}}

\setlength{\itemsep}{-1mm}


\bibitem{AGHKT}
A.~Anandkumar, R.~Ge, D.~Hsu, S.~Kakade and M.~Telgarsky:
{\em Tensor decompositions for learning latent variable models},
Journal of Machine Learning Research {\bf 15} (2014) 2773--2832.

\bibitem{BertiniSoftware}
D.J.~Bates, J.D.~Hauenstein, A.J.~Sommese and C.W.~Wampler.
\newblock \emph{Bertini: Software for numerical algebraic geometry}.
\newblock Available at \url{bertini.nd.edu}.

\bibitem{BCMT}
J.~Brachat, P.~Comon, B.~Mourrain and E.~Tsigaridas:
{\em Symmetric tensor decomposition},
Linear Algebra and its Applications {\bf 433} (2010) 851--872.

\bibitem{BrOt} M.~Brambilla and G.~Ottaviani:
{\em On the Alexander-Hirschowitz theorem}, J. Pure Appl. Algebra 
{\bf 212} (2008) 1229--1251.

\bibitem{CStr} J.~Cahill and N.~Strawn: {\em Algebraic geometry and finite frames},
 Finite frames, 141--170, Appl.~Numer.~Harmon.~Anal., Birkh\"auser/Springer, New York, 2013.

\bibitem{CMS} J.~Cahill, D.~Mixon and N.~Strawn:
{\em Connectivity and irreducibility of algebraic 
varieties of finite unit norm tight frames}, {\tt arXiv:1311.4748}.

\bibitem{CStu} D.~Cartwright and B.~Sturmfels:
{\em The number of eigenvalues of a tensor},
Linear Algebra Appl.~{\bf 438} (2013) 942--952.

\bibitem{CFMPS} P.~Casazza, M.~Fickus, D.~Mixon, J.~Peterson and I.~Smalyanau:
{\em Every Hilbert space frame has a Naimark complement},
 J. Math. Anal. Appl. {\bf 406} (2013) 111--119. 
 
\bibitem{CKP} P. ~Casazza, G.~Kutyniok and F.~Philipp:
{\em Introduction to finite frame theory}, Finite frames, 1--53, Appl. Numer. Harmon. Anal., 
Birkh\"auser/Springer, New York, 2013. 
   
 \bibitem{CGLM} P.~Comon, G.~Golub, L.-H.~Lim and B.~Mourrain:
 {\em  Symmetric tensors and symmetric tensor rank},
 SIAM J. Matrix Anal. Appl. {\bf 30} (2008) 1254--1279
 
 \bibitem{CLO} 
D.~Cox, J.~Little and D.~O'Shea:  
{\em  Ideals, Varieties, and Algorithms.
 An Introduction to Computational Algebraic Geometry and Commutative Algebra},
 Undergraduate Texts in Mathematics, Springer, New York, 1992


 \bibitem{DS} K.~Dykema and N.~Strawn:
{\em Manifold structure of spaces of spherical tight frames},
Int. J. Pure Appl. Math. {\bf 28} (2006) 217--256. 

\bibitem{M2} D.~Grayson and M.~Stillman:
{\em Macaulay2, a software system for research in algebraic geometry}, available at
{\tt www.math.uiuc.edu/Macaulay2/}.

\bibitem{H} J.~Harris:
{\em Algebraic Geometry: A First Course},
Graduate Texts in Mathematics, {\bf 133}, Springer-Verlag, New York, 1995.

\bibitem{HIS}
 J.D.~Hauenstein, C.~Ikenmeyer and J.M.~Landsberg:
 {\em  Equations for lower bounds on border rank},
   Experimental Mathematics {\bf 22} (2013) 373--383.

\bibitem{HOOS}
J.D.~Hauenstein, L.~Oeding, G.~Ottaviani and A.J.~Sommese:
{\em Homotopy techniques for tensor decomposition and perfect identifiability},
\arxiv{1501.00090}.

\bibitem{HS} J.D.~Hauenstein and A.J.~Sommese:
{\em Witness sets of projections}, Appl. Math.~Comput.~{\bf 217} (2010) 3349--3354.

\bibitem{Lan} J.M.~Landsberg:
{\em Tensors: Geometry and Applications},
 Graduate Studies in Mathematics, 128,
 American Mathematical Society, Providence, RI, 2012.
 
\bibitem{Murai} S.~Murai:
{\em Generic initial ideals and squeezed spheres},
Advances in Math.~{\bf 214} (2007) 701--729.

\bibitem{OO}
L.~Oeding and G.~Ottaviani:
{\em Eigenvectors of tensors and algorithms for Waring decomposition},
 J.~Symbolic Computation {\bf 54} (2013) 9--35.

\bibitem{Rob} E.~Robeva:
{\em Orthogonal decomposition of symmetric tensors},
{\tt arXiv:1409.6685}.

\bibitem{SSS} R.~Sanyal, F.~Sottile and B.~Sturmfels:
{\em Orbitopes}, Mathematika {\bf 57} (2011) 275--314.

\bibitem{SidmanSullivant} J.~Sidman and S.~Sullivant: 
{\em Prolongations and computational algebra}, Canad. J. Math. {\bf 4} (2009) 930--949.

\bibitem{Stanley} R.~Stanley:
{\em Combinatorics and Commutative Algebra},
second edition, Progress in Mathematics, 
{\bf 41}, Birkh\"auser, Boston, MA, 1996.

\bibitem{SVW} A.J.~Sommese, J.~Verschelde, and C.W.~Wampler:
{\em  Symmetric functions applied
 to decomposing solution sets of polynomial systems},
 SIAM J.~Numer. Anal.~{\bf 40} (2002) 2026--2046.

\bibitem{Str} N.~Strawn: {\em Finite frame varieties: nonsingular points, tangent spaces, and explicit local parameterizations},
 J.~Fourier Anal. Appl. {\bf 17.5} (2011) 821--853.

\end{thebibliography}
\end{document}